\documentclass{amsart}

\usepackage[foot]{amsaddr}
\let\oldtocsection=\tocsection

\let\oldtocsubsection=\tocsubsection 

\let\oldtocsubsubsection=\tocsubsubsection

\renewcommand{\tocsection}[2]{\vspace{0.5em}\hspace{0em}\oldtocsection{#1}{#2}}
\renewcommand{\tocsubsection}[2]{\vspace{0.5em}\hspace{1em}\oldtocsubsection{#1}{#2}}
\renewcommand{\tocsubsubsection}[2]{\vspace{0.5em}\hspace{2em}\oldtocsubsubsection{#1}{#2}}
\usepackage{graphicx,cancel,xcolor,hyperref,comment,graphicx,geometry}

\usepackage{tikz}
\usetikzlibrary{decorations.pathreplacing}
\usepackage{graphicx,url,etoolbox}
\usepackage{lipsum}
\usepackage{dsfont}
\usepackage{amsmath}

\usepackage{verbatim}

\usepackage{fancyhdr}
\usepackage{lastpage}

\newtheorem{theoreme}{Theorem}[section]

\theoremstyle{definition}

\numberwithin{equation}{section}

\renewenvironment{proof}{{\bfseries \noindent Proof.}}{\demo}
\newcommand\xqed[1]{%
	\leavevmode\unskip\penalty9999 \hbox{}\nobreak\hfill
	\quad\hbox{#1}}
\newcommand\demo{\xqed{$\square$}}

\hypersetup{bookmarks, colorlinks, urlcolor=blue, citecolor=blue, linkcolor=blue, hyperfigures, pagebackref,
	pdfcreator=LaTeX, breaklinks=true, pdfpagelayout=SinglePage, bookmarksopen=true,bookmarksopenlevel=2}

\let\originallesssim\lesssim
\let\originalgtrsim\gtrsim

\DeclareRobustCommand{\lesssim}{%
	\mathrel{\mathpalette\lowersim\originallesssim}%
}
\DeclareRobustCommand{\gtrsim}{%
	\mathrel{\mathpalette\lowersim\originalgtrsim}%
}
\makeatletter
\newcommand{\lowersim}[2]{%
	\sbox\z@{$#1<$}%
	\raisebox{-\dimexpr\height-\ht\z@}{$\m@th#1#2$}%
}
\makeatother

\usepackage{comment}

\def\R{\mathbb R}
\def\Z{\mathbb Z}
\def\N{\mathbb N}

\def\C{\mathbb C}
\def\HH{\mathcal H}

\def\AA{\mathcal A}

\def\la {{\lambda}}

\newcommand {\nc}   {\newcommand}
\nc {\be}   {\begin{equation}} \nc {\ee}   {\end{equation}} \nc
{\beq}  {\begin{eqnarray}} \nc {\eeq}  {\end{eqnarray}} \nc {\beqs}
{\begin{eqnarray*}} \nc {\eeqs} {\end{eqnarray*}}
\def\edc{\end{document}}

\providecommand{\abs}[1]{\lvert#1\rvert}%absolute value
\numberwithin{equation}{section}

\theoremstyle{Thm}
\newtheorem{Thm}{Theorem}[section]
\newtheorem{lem}{Lemma}[section]
\newtheorem{prop}{Proposition}[section]

\newtheorem{rk}{Remark}[section]
\definecolor{carnelian}{rgb}{0.7, 0.11, 0.11}
\definecolor{carmine}{rgb}{0.59, 0.0, 0.09}
\definecolor{burgundy}{rgb}{0.5, 0.0, 0.13}
\definecolor{darkmidnightblue}{rgb}{0.0, 0.2, 0.4}
\definecolor{dimgray}{rgb}{0.75, 0.75, 0.75}
\definecolor{palecarmine}{rgb}{0.69, 0.25, 0.21}
\numberwithin{dummy}{section}

\numberwithin{equation}{section}
\def\AA{\mathcal A}
\def\HH{\mathbf{\mathcal H}}

\newcommand{\h}{\mathsf{h}}
\newcommand{\f}{\mathsf{f}}

\newcommand{\g}{\mathsf{f}_3}
\newcommand{\q}{\mathsf{g}}

\newcommand{\intdx}{\int_{0}^{L} }

\newcommand{\intdnb}{\int_{0 }^{\beta } }

\providecommand{\abs}[1]{\lvert#1\rvert}%absolute value
\begin{document}
	\title[\fontsize{7}{9}\selectfont  ]{The influence of the physical coefficients of a Bresse system with one singular local viscous damping in the longitudinal displacement on its stabilization
	}
\author{Mohammad Akil$^{1}$}
\author{Haidar Badawi$^{2}$}
\address{$^1$ Universit\'e Savoie Mont Blanc, Laboratoire LAMA, Chamb\'ery-France}
\address{$^2$ Universit\'e Polytechnique Hauts-de-France (UPHF-LAMAV),
	Valenciennes, France}
%\address{$^3$Lebanese University, Faculty of sciences 1, Khawarizmi Laboratory of  Mathematics and Applications-KALMA, Hadath-Beirut, Lebanon.}
\email{mohammad.akil@univ-smb.fr, Haidar.Badawi@etu.uphf.fr.}
%%%%%%%%%%%%%%
%%%%%%%%%%%%%%%
%abstract%%%%%%55
%%%%%%%%%%%%%%%%%
\begin{abstract}
	In this paper, we investigate the stabilization of a linear Bresse system with one singular local frictional damping  acting in the longitudinal displacement, under fully Dirichlet boundary conditions. First, we prove the strong stability of our system. Next, using a frequency domain approach combined with the multiplier method, we establish the exponential stability of the solution if and only if the three waves have the same speed of propagation. On the contrary,  we prove that the energy of our system  decays polynomially with rates $t^{-1}$ or $t^{-\frac{1}{2}}$.
	\end{abstract}
\maketitle
\pagenumbering{roman}
\maketitle
\tableofcontents
%\clearpage
\pagenumbering{arabic}
\setcounter{page}{1}
\newpage
%%Introduction%%%%%
%%%%%%%%%%%%%%%%%%%%%5
%%%%%%%%%%%%%%%%%%5
\section{Introduction}
\noindent In this paper, we investigate the stability of Bresse system with one discontinuous local frictional damping in the longitudinal displacement. More precisely, we consider the following system:
\begin{equation}\label{p3-sysorig}
\left\{	\begin{array}{llll}
	\displaystyle \rho_1 \varphi_{tt}-k_1 (\varphi_x+\psi+lw)_x -lk_3 (w_x-l\varphi)=0, &(x,t)\in  (0,L) \times (0,\infty),&\vspace{0.15cm}\\
	\displaystyle \rho_2 \psi_{tt}-k_2 \psi_{xx} +k_1 (\varphi_x +\psi+lw)=0,&(x,t)\in  (0,L) \times (0,\infty),&\vspace{0.15cm}\\
\displaystyle \rho_1 w_{tt}-k_3 (w_x-l\varphi) _x+lk_1 (\varphi_x+\psi+lw)+a(x)w_t=0, &(x,t)\in  (0,L) \times (0,\infty),
	\end{array}
	\right.
\end{equation}
with the following Dirichlet boundary conditions 
\begin{equation}
	\varphi(0,t)=\varphi(L,t)=\psi(0,t)=\psi(L,t)=w (0,t)=w(L,t)=0, \ \ t>0.
\end{equation}
and the following initial conditions
\begin{equation}\label{p3-initialcond}
\left\{\begin{array}{lll}
\displaystyle 	\varphi(x,0)=\varphi_0(x), \ \varphi_t (x,0)=\varphi_1(x), \ \psi(x,0)=\psi_0(x),  \ x\in(0,L),\vspace{0.15cm}\\
\displaystyle  \psi_t(x,0)=\psi_1(x), \ w(x,0)=w_0(x), \ w_t (x,0)=w_1(x), \ x\in(0,L),
	\end{array}
	\right.
\end{equation}
where $\rho_1, \rho_2, k_1, k_2, k_3, l $ and $L$ are  positive real numbers. We suppose that there exists $0<\beta<L$ and a positive constant $a_0$ such that 
\begin{equation}\label{p3-a}
	a(x)=\left\{\begin{array}{lll}
	a_0 & \text{if} & x\in (0,\beta),\vspace{0.15cm}\\
		0 & \text{if} & x\in (\beta,L).
	\end{array}
	\right.
\end{equation}
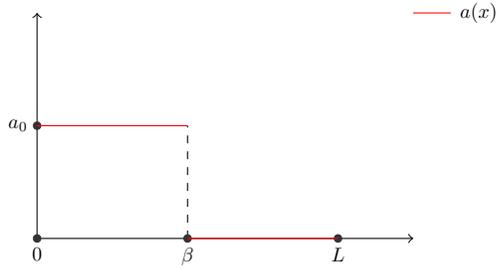
\begin{figure}[h!]
	\begin{center}
		\begin{tikzpicture}
		\draw[->](0,0)--(5,0);
		\draw[->](0,0)--(0,3);

		\draw[dashed](2,0)--(2,1.5);
		%	\draw[dashed](4,0)--(4,2);

		\node[black,below] at (2,0){\scalebox{0.75}{$\beta$}};
		\node at (2,0) [circle, scale=0.3, draw=black!80,fill=black!80] {};

		\node[black,below] at (4,0){\scalebox{0.75}{$L $}};
		\node at (4,0) [circle, scale=0.3, draw=black!80,fill=black!80] {};

		%	\node[black,below] at (5,0){\scalebox{0.75}{$L$}};
		%	\node at (5,0) [circle, scale=0.3, draw=black!80,fill=black!80] {};
		
		\node[black,below] at (0,0){\scalebox{0.75}{$0$}};
		\node at (0,0) [circle, scale=0.3, draw=black!80,fill=black!80] {};
		
		\node at (0,1.5) [circle, scale=0.3, draw=black!80,fill=black!80] {};
		\node[black,left] at (0,1.5){\scalebox{0.75}{$a_0$}};

		\node[black,right] at (5.5,3){\scalebox{0.75}{$a(x)$}};

		%%keys %%
		\draw[-,red](5,3)--(5.5,3);
		%%%%%b(x)%%%%%%%%%%%%%%
		\draw[-,red](0,1.5)--(2,1.5);
		\draw[-,red](2,0)--(4,0);

		\end{tikzpicture}
	\end{center}
	\caption{Geometric description of the function $a(x)$.}\label{p3-Fig1}
\end{figure}\\

 The notion of indirect damping mechanisms has been introduced by Russell in \cite{Russell01} and since this time, it retains the attention of many authors. In particular, the fact that only one equation of the coupled system is damped refers to the so-called class of "indirect" stabilization problems initiated and studied in \cite{Alabau04,Alabau02,Alabau05} and further studied by many authors, see for instance  \cite{Alabau06,LiuRao01,ZhangZuazua01} and the rich references therein.\\\linebreak
 
The Bresse system  is a model for arched beams, see \cite[Chap. 6]{LagneseLeugering01}. It can be expressed by the equations of motion:
\begin{equation}\label{p4-1.5}
\left\{\begin{array}{lll}
\displaystyle	\rho_1 \varphi_{tt}=Q_x +lN,\vspace{0.15cm}\\
\displaystyle	\rho_2 \psi_{tt}=M_x -Q,\vspace{0.15cm}\\
\displaystyle	\rho_1 w_{tt}=N_x -lQ-a(x)\omega_t,
\end{array}
\right.
\end{equation}
where  $N=k_3 (w_x-l\varphi)$ is the axial force, $Q=k_1 (\varphi_x+\psi+lw)$ is the shear force, and $M=k_2 \psi_{x}$ is the bending moment.
The functions $\varphi$, $\psi$, and $w$ are respectively the vertical, shear angle, and longitudinal displacements.
Here $\rho_1 =\rho A$, $\rho_2 =\rho I$, $k_1 =kGA$, $k_3 =EA$, $k_2 =EI$ and $l=R^{-1}$, in which  $\rho$ is the density of the material, 
$E$ the modulus of the elasticity, $G$ the shear modulus, $k$ the shear factor,  $A$ the cross-sectional area,  $I$ the second moment of area of the cross section, $R$ the radius of the curvature, and $l$ the curvature. Moreover, $F_1$, $F_2$, and $F_3$ are the external forces.
\\ \linebreak  
There are several publications concerning the stabilization of Bresse system with frictional or another kinds of damping (see  \cite{Wehbe08}, \cite{doi:10.1002/mma.6070}, \cite{ALABAUBOUSSOUIRA2011481}, \cite{https://doi.org/10.1002/mma.3115}, \cite{deLima2018}, \cite{ElArwadi2019}, \cite{doi:10.1080/00036811.2018.1520982}, \cite{FATORI2012600}, \cite{10.1093/imamat/hxq038}, \cite{Guesmia2017}, \cite{GHOUL20171870}, \cite{doi:10.1002/mma.3228}, \cite{Rivera2019},  \cite{RaoLiu03}, \cite{Wehbe03}, \cite{Wehbe02}, \cite{SORIANO2014369} and \cite{Wehbe01}). We note that by neglecting $w$ ($l \to 0$) in \eqref{p4-1.5}, the Bresse system  reduces to the following conservative Timoshenko system: 
\begin{equation*}
\begin{array}{lll}
\rho_1 \varphi_{tt}-k_1 (\varphi_x+\psi)_x=0,\vspace{0.25cm}\\
\rho_2 \psi_{tt}-k_2 \psi_{xx}+k_1 (\varphi_x+\psi)=0.
\end{array}
\end{equation*}
There are also several publications concerning the stabilization of Timoshenko system  with different kinds of damping (see  \cite{Akil2020}, \cite{Alabau-Boussouira2007}, \cite{BASSAM20151177}, \cite{doi:10.1002/zamm.201500172} and  \cite{doi:10.1080/00036810903156149}).\\ \linebreak
Among this vast literature let us recall some specific results on the Bresse systems.\\\linebreak 

In 2010, Wehbe and Youssef in \cite{Wehbe01} studied the stability of an elastic Bresse system with two locally distributed frictional dampings on shear angle and longitudinal displacements, under fully Dirichlet or Dirichlet-Neumann-Neumann boundary conditions; they showed that the system is exponential stable if and only if the equations of the vertical displacement and rotation angle have the same wave speeds of propagation. In  case that the wave speeds of the equations are different, they obtained a polynomial decay rate. In 2011, Alabau {\it et al.}  in \cite{ALABAUBOUSSOUIRA2011481} studied the stability of a Bresse system with one frictional damping on the shear angle displacement, under fully Dirichlet or Dirichlet-Neumann-Neumann conditions; they showed that the system is exponential stable if and only if the three equations have the same wave speeds of propagation. On the contrary, they proved that the solution of the system decays polynomially  with rates $t^{-3+\epsilon}$ or $t^{-6+\epsilon}$, where $\epsilon>0$. In 2012, Noun and Wehbe in \cite{Wehbe02} studied the stability of a Bresse system with one local frictional damping on the shear angle displacement, under fully Dirichlet or Dirichlet-Neumann-Neumann boundary conditions; they showed that the system is exponential stable if and only if the three equations have the same wave speeds of propagation. On the contrary,  they proved that the energy of the system  decays polynomially with different rates. In 2013, Soriano {\it et al.} in \cite{SORIANO2014369} studied the asymptotic stability of a Bresse system with a nonlinear frictional damping on the shear angle displacement, and nonlinear localized damping in the vertical and longitudinal displacement; they proved the asymptotic stability of the system.  In 2015, Alves {\it et al} in  \cite{https://doi.org/10.1002/mma.3115}  studied the stability of a Bresse system with two frictional dampings on vertical and longitudinal displacements, under Dirichlet-Neumann-Neumann boundary conditions; they showed that the system is exponential stable if and only if the equations of the vertical displacement and longitudinal displacement have the same wave speeds of propagation. In  case that the wave speeds of the equations are different, they proved that the solution decays polynomially to zero with optimal decay rate.
 In 2018, Afilal {\it et al.} in \cite{doi:10.1002/mma.6070} studied the stability of a Bresse system with global frictional damping in the longitudinal displacement,  under mixed boundary conditions of the form
 $$
\left\{ \begin{array}{lll}
 \displaystyle \varphi(0,t)=\psi_x(0,t)=w_x (0,t)=0, \quad \text{in} \ \ (0,\infty),\vspace{0.15cm}\\  \displaystyle \varphi_x(1,t)=\psi(1,t)=w(1,t)=0,\ \, \quad \text{in} \ \ (0,\infty),
 \end{array}
 \right.
 $$ 
 they assumed that the curvature $l$ satisfies
\begin{equation*}\label{p3-1.7*}
l\neq \frac{\pi}{2}+m\pi, \quad \forall \,  m\in \N \ \ \text{and} \ \   l^2\neq \frac{\rho_2k_3+\rho_1k_2}{\rho_2 k_3}\left(\frac{\pi}{2}+m\pi\right)^2+\frac{\rho_1k_1}{\rho_2(k_1+k_3)}, \ \ \forall m\in \Z;
\end{equation*} 
they showed that the system is exponential stable if and only if the equations have the same wave speeds of propagation. In case that the wave speeds of the equations are different, they established a polynomial energy decay rate of order $t^{-\frac{1}{4}}$.\\\linebreak

 In this paper, we extend the results in \cite{doi:10.1002/mma.6070}, by assuming that the frictional damping is locally distributed in the longitudinal displacement, under fully Dirichlet boundary conditions and without any condition on the curvature $l$, we also improve the polynomial energy decay rate.\\\linebreak 
 
  But to the best of our knowledge, it seems that  no result in the literature exists concerning the case of Bresse system with one discontinuous local frictional damping in the longitudinal displacement, especially under fully Dirichlet boundary conditions and without any condition on the curvature $l$. The goal of the present paper is to fill this gap by studying the stability of system \eqref{p3-sysorig}-\eqref{p3-initialcond}.\\\linebreak
    
  This paper is organized as follows: In Section \ref{p3-WPS}, we prove the well-posedness of our system by using semigroup approach. In Section \ref{p3-sec3}, we show the strong stability of our system. Finally, in Section \ref{p3-secpoly}, by using the frequency domain approach  combining with a specific  multiplier method,  we establish the exponential stability of the solution if and only if the three waves have the same speed of propagation (i.e., $\frac{k_1}{\rho_1}=\frac{k_2}{\rho_2}$ and $k_1=k_3$). On the contrary, we prove that the energy of our system  decays polynomially with the rates:
  \begin{equation*}
  \left\{	\begin{array}{lll}
  \displaystyle	t^{-1} \quad \text{if} \quad \frac{k_1}{\rho_1}=\frac{k_2}{\rho_2} \ \ \text{and} \ \ k_1\neq k_3,\vspace{0.15cm}\\
 \displaystyle 		t^{-\frac{1}{2}} \quad \text{if} \quad \frac{k_1}{\rho_1}\neq\frac{k_2}{\rho_2} .
  	\end{array}
  	\right.
  \end{equation*}

\section{Well-posedness of the system}\label{p3-WPS}
\noindent In this section,  we will establish the well-posedness of system \eqref{p3-sysorig}-\eqref{p3-initialcond} by using semigroup approach.
The energy of system \eqref{p3-sysorig}-\eqref{p3-initialcond} is given by 
$$
\begin{array}{lll}
\displaystyle E(t)=\frac{1}{2}\intdx \left(\rho_1 \left|\varphi_t\right|^2+\rho_2|\psi_t|^2 +\rho_1 |w_t|^2 +k_1 |\varphi_x+\psi+lw|^2 +k_2 |\psi_x|^2 +k_3 |w_x -l\varphi|^2\right) dx.
\end{array}  
$$
Let $(\varphi,\varphi_t,\psi,\psi_t,w,w_t)$ be a regular solution of system \eqref{p3-sysorig}-\eqref{p3-initialcond}. Multiplying the equations in \eqref{p3-sysorig} by $\overline{\varphi_t}$, $\overline{\psi_t}$ and $\overline{w_t}$ respectively, Then using the boundary conditions \eqref{p3-bc} and the definition of $a(x)$ (see \eqref{p3-a} and Figure \ref{p3-Fig1}), we obtain  
\begin{equation}\label{p3-2.12}
	E^{\prime }(t)=-\intdx a(x) |w_{t}|^2 dx=-a_0\intdnb  |w_{t}|^2 dx \leq 0.
\end{equation}
From \eqref{p3-2.12}, system \eqref{p3-sysorig}-\eqref{p3-initialcond} is dissipative in the sense that its energy is non-increasing with respect to time. Now, we define the following Hilbert space $\HH$ by:
$$
\HH:=\left( H^1_0(0,L)\times L^2 (0,L)\right)^3.
$$
The Hilbert space $\HH$ is equipped with the following inner product
$$
\begin{array}{lll}
\displaystyle (U,U^1 )_{\HH}=\intdx \left\{ k_1(v^{1}_x+v^3+lv^5)\overline{(\widetilde{v^{1}_x}+\widetilde{v^3}+l\widetilde{v^5})}+\rho_1 v^2\overline{\widetilde{v^2}}+k_2 v^3_x \overline{\widetilde{v^3_x}} +\rho_2 v^4\overline{\widetilde{v^4}}\right.\vspace{0.25cm}
\\
\hspace{3cm}\displaystyle \left. +\, k_3 (v^5_x-lv^1)(\overline{\widetilde{v^5_x}-l\widetilde{v^1}})dx+\rho_1 v^6\overline{\widetilde{v^6}}\right\}dx,
\end{array}
$$
where  $U=(v^1,v^2,v^3,v^4,v^5,v^6)^{\top}\in \HH$ and $\widetilde{U} =(\widetilde{v^1}, \widetilde{v^1},\widetilde{v^2},\widetilde{v^3},\widetilde{v^4},\widetilde{v^5}, \widetilde{v^6} )^{\top}\in\HH$. Now, we define the linear unbounded  operator $\AA:D(\AA)\subset \HH\longmapsto \HH$  by:
\begin{equation}
D(\AA)=\left[\left(H^2(0,L)\cap H^1_0(0,L)\right) \times H^1_0 (0,L)\right]^3
\end{equation}and 
\begin{equation}\label{p2-op}
\AA\begin{pmatrix}
v^1\\v^2\\v^3\\v^4\\v^5\\v^6
\end{pmatrix}=
\begin{pmatrix} 
v^2\\\displaystyle  \frac{k_1}{\rho_1}(v^{1}_x+v^3+lv^5 )_x+\frac{lk_3}{\rho_1}(v^{5}_x-lv^1 )\\\displaystyle v^4\\\displaystyle \frac{k_2}{\rho_2}v^{3}_{xx}  -\frac{k_1}{\rho_2}(v^{1}_x+v^3+lv^5 )\\\displaystyle v^6 \vspace{0.15cm}
\\\displaystyle \frac{k_3}{\rho_1}(v^{5}_x-lv^1 ) _x-\frac{lk_1}{\rho_1}(v^{1}_x+v^3+lv^5)-\frac{a(x)}{\rho_1 }v^6
\end{pmatrix},
\end{equation} 
for all $U=(v^1,v^2,v^3,v^4,v^5,v^6)^\top \in D(\AA)$.\\ 
In this sequel, $\|\cdot\|$ will denote the usual norm of $L^2 (0,L)$.\\
\linebreak 
Now, if $U=(\varphi, \varphi_t,\psi,\psi_t,w,w_t)^{\top}$, then system \eqref{p3-sysorig}-\eqref{p3-initialcond} can be written as the following first order evolution equation 
\begin{equation}\label{p3-firstevo}
U_t =\AA U , \quad U(0)=U_0,
\end{equation}
where  $U_0 =(\varphi_0 ,\varphi_1,\psi_0,\psi_1,w_0,w_1 )^{\top}\in \HH$.
\begin{prop}\label{p3-mdissip}
	{\rm The unbounded linear operator $\AA$ is m-dissipative in the Hilbert space $\HH$.}
\end{prop}
\begin{proof}
For all $U=(v^1,v^2,v^3,v^4,v^5,v^6)^\top \in D(\AA)$, we have
\begin{equation}\label{p3-Reau}
	\Re (\AA U,U)_{\HH}=-\intdx a(x)\left|v^6\right|^2 dx = -a_0 \intdnb \left|v^6\right|^2 dx \leq 0.
\end{equation}
which implies that $\AA$ is dissipative. Let us prove that $\AA$ is maximal. For this aim, let $F=(f^1,f^2,f^3,f^4,f^5,f^6)^{\top}\in\HH$, we look for $U=(v^1,v^2,v^3,v^4,v^5,v^6)^{\top}\in D(\AA)$ unique solution of 
\begin{equation}\label{p3-AU=F}
	-\AA U=F.
\end{equation}
Detailing \eqref{p3-AU=F}, we obtain 
\begin{eqnarray}
-v^2&=&f^1,\label{p3-f1}\\
-k_1\left(v^1_x+v^3+lv^5 \right)_x -lk_3(v^5_x-lv^1)&=&\rho_1f^2,\label{p3-f2}\\
-v^4&=&f^3,\label{p3-f3}\\
	-k_2v^{3}_{xx} +k_1(v^1_x+v^3+lv^5)&=&\rho_2 f^4,\label{p3-f4}\\
	-v^6&=&f^5,\label{p3-f5}\\
-k_3\left(v^5_x-lv^1\right)_x +lk_1(v^1_x+v^3+lv^5)+a(x)v^6 &=&\rho_1 f^6,\label{p3-f6}
\end{eqnarray}
with the following boundary conditions
\begin{equation}\label{p3-bc}
	v^1 (0)=v^1(L)=v^3(0)=v^3(L)=v^5(0)=v^5(L)=0.
\end{equation}
Inserting \eqref{p3-f5} in \eqref{p3-f6}, we obtain
\begin{eqnarray}
%-\frac{k_1}{\rho_1}\left(v^1_x+v^3+lv^5 \right)_x %-\frac{lk_3}{\rho_1}(v^5_x-lv^1)&=&f^2,\label{p3-f7}\\
%-\frac{k_2}{\rho_2}v^{3}_{xx} %+\frac{lk_1}{\rho_1}(v^1_x+v^3+lv^5)&=&f^4,\label{p3-f8}\\
-k_3\left(v^5_x-lv^1\right)_x +lk_1(v^1_x+v^3+lv^5)&=&\rho_1f^6+a(x)f^5.\label{p3-f9}
\end{eqnarray}
	Let $(\phi^1 ,\phi^2,\phi^3) \in \left(H^{1}_0 (0,L)\right)^3$. Multiplying  \eqref{p3-f2}, \eqref{p3-f4}  and \eqref{p3-f9} by $\overline{\phi^1}$, $\overline{\phi^2}$ and $\overline{\phi^3}$ respectively, integrating over $(0,L)$, then using formal integrations by parts, we obtain
	\begin{equation}\label{p3-vf}
	\mathcal{B}((v^1,v^3,v^5),(\phi^1,\phi^2,\phi^3))=\mathcal{L}((\phi^1,\phi^2,\phi^3)), \ \ \forall (\phi^1,\phi^2,\phi^3)\in \left(H^1_0 (0,L)\right)^3,
	\end{equation}
	where
$$
\begin{array}{lll}
&&\displaystyle 	\mathcal{B}((v^1,v^3,v^5),(\phi^1,\phi^2,\phi^3))=\displaystyle k_1\intdx (v^1_x+v^3 +lv^5)\overline{\phi^1_x }dx -lk_3\intdx (v^5_x -lv^1)\overline{\phi^1}dx \vspace{0.25cm}\\
&&\displaystyle  +\, k_2\intdx v^3_x \overline{\phi^2_x}dx  +k_1\int_{0}^L (v^1_x +v^3 +lv^5)\overline{\phi^2}dx 
  +k_3\intdx (v^5_x -lv^1)\overline{\phi^3_x}dx +lk_1\intdx (v^1_x+v^3 +lv^5)\overline{\phi^3}dx
\end{array}
$$
and
$$
\begin{array}{lll}
&&\displaystyle \mathcal{L}((\phi^1,\phi^2,\phi^3))=\rho_1\intdx f^2 \overline{\phi^1}dx+\rho_2\intdx f^4 \overline{\phi^2}dx+\rho_1\intdx f^6 \overline{\phi^3}dx+\intdx a(x)f^5 \overline{\phi^3}dx.
\end{array}
$$
It is easy to see that,  $\mathcal{B}$ is a sesquilinear, continuous and coercive form on $\left( H^{1}_0 (0,L)\right)^3  \times \left(H^{1}_0 (0,L)\right)^3 $ and $\mathcal{L}$ is a antilinear and continuous form on $\left( H^{1}_0 (0,L)\right)^3 $. Then, it follows by Lax-Milgram theorem that \eqref{p3-vf} admits a unique solution $(v^1,v^3,v^5)\in \left( H^{1}_0 (0,L)\right)^3 $. By taking test-functions 
$(\phi^1,\phi^2,\phi^3)\in \left(\mathcal{D} (0,L)\right)^3$, we see that
(\eqref{p3-f2}, \eqref{p3-f4}, \eqref{p3-f9}, \eqref{p3-bc}) hold in the distributional sense, from which we deduce that    $(v^1,v^3,v^5)\in  \left(H^2 (0,L)\cap H^{1}_0 (0,L)\right)^3$. Consequently,  $U=(v^1,-f^1,v^3,-f^3,v^5,-f^5)^{\top} \in D(\AA) $ is a unique solution of \eqref{p3-AU=F}. Then, $\mathcal{A}$ is an isomorphism and since $\rho\left(\mathcal{A}\right)$ is open set of $\mathbb{C}$ (see Theorem 6.7 (Chapter III) in \cite{Kato01}),  we easily get $R(\lambda I -\mathcal{A}) = {\mathcal{H}}$ for a sufficiently small $\lambda>0 $. This, together with the dissipativeness of $\mathcal{A}$, imply that   $D\left(\mathcal{A}\right)$ is dense in ${\mathcal{H}}$   and that $\mathcal{A}$ is m-dissipative in ${\mathcal{H}}$ (see Theorems 4.5, 4.6 in  \cite{Pazy01}). The proof is thus complete.
\end{proof}\\\linebreak
According to Lumer-Philips theorem (see \cite{Pazy01}), Proposition \ref{p3-mdissip} implies that the operator $\AA$ generates a $C_{0}$-semigroup of contractions $e^{t\AA}$ in $\HH$ which gives the well-posedness of \eqref{p3-firstevo}. Then, we have the following result:
\begin{Thm}{\rm
For all $U_0 \in \HH$,  system \eqref{p3-firstevo} admits a unique weak solution $$U(t)=e^{t\AA}U_0 \in C^0 (\R^+ ,\HH).
	$$ Moreover, if $U_0 \in D(\AA)$, then the system \eqref{p3-firstevo} admits a unique strong solution $$U(t)=e^{t\AA}U_0 \in C^0 (\R^+ ,D(\AA))\cap C^1 (\R^+ ,\HH).$$}
\end{Thm}

\section{Strong Stability}\label{p3-sec3}
\noindent In this section, we will prove the strong stability of  system \eqref{p3-sysorig}-\eqref{p3-initialcond}. The main result of this section is the following theorem.
\begin{theoreme}\label{p3-strongthm2}
	{\rm	The $C_0-$semigroup of contraction $\left(e^{t\AA}\right)_{t\geq 0}$ is strongly stable in $\HH$; i.e., for all $U_0\in \HH$, the solution of \eqref{p3-firstevo} satisfies 
		$$
		\lim_{t\rightarrow +\infty}\|e^{t\AA}U_0\|_{\HH}=0.
		$$}
\end{theoreme}
\begin{proof} Since the resolvent of $\AA$ is compact in $\HH$, then according to Arendt-Batty theorem see (Page 837 in \cite{Arendt01}), system \eqref{p3-sysorig}-\eqref{p3-initialcond} is strongly stable if and only if $\AA$ doesn't have pure imaginary eigenvalues that is $\sigma (\AA)\cap i\R=\emptyset$. 
	From Proposition \ref{p3-mdissip}, we have $0\in \rho (\AA)$. We still need to show that $\sigma(\AA) \cap i\R^* =\emptyset$. For this aim, suppose by contradiction that there exists a real number $\la\neq0$ and $U=(v^1,v^2,v^3,v^4,v^5,v^6)^{\top}\in D(\AA)\backslash\{0\}$ such that 
	\begin{equation}\label{p2-AU=ilaU}
	\AA U=i\la U.
	\end{equation}Equivalently, we have the following system
	\begin{eqnarray}
	v^2=i\la v^1 \label{p3-f1ker},
	\\\displaystyle  \frac{k_1}{\rho_1}(v^{1}_x+v^3+lv^5 )_x+\frac{lk_3}{\rho_1}(v^{5}_x-lv^1 )=i\la v^2\label{p3-f2ker},
	\\\displaystyle v^4=i\la v^3\label{p3-f3ker},
	\\\displaystyle \frac{k_2}{\rho_2}v^{3}_{xx}  -\frac{k_1}{\rho_2}(v^{1}_x+v^3+lv^5 )=i\la v^4\label{p3-f4ker},
	\\\displaystyle v^6=i\la v^5 \label{p3-f5ker},\vspace{0.15cm}
	\\\displaystyle \frac{k_3}{\rho_1}(v^{5}_x-lv^1 )_x-\frac{lk_1}{\rho_1}(v^{1}_x+v^3+lv^5)+\frac{a(x)}{\rho_1}v^6=i\la v^6
	\vspace{0.15cm}\label{p3-f6ker}.
	\end{eqnarray} 
	From  \eqref{p3-Reau} and \eqref{p2-AU=ilaU}, we obtain 
	\begin{equation}\label{p2-dissi=0}
	0=\Re \left(i\la U,U\right)_\HH=\Re\left(\AA U,U\right)_{\HH}=-\intdx a(x)\left|v^6\right|^2 dx =-a_0\intdnb \left|v^6\right|^2 dx.
	\end{equation}
	Thus, we have
	\begin{equation}\label{p3-2.43}
v^6=0 \ \ \text{in} \ \ (0,\beta).
	\end{equation}
	From \eqref{p3-f5ker}, \eqref{p3-2.43} and the fact that $\la \neq 0$, we get 
	\begin{equation}\label{p3-3.10}
	v^5=0 \ \ \text{in} \ \ (0,\beta).
	\end{equation}
	Now, from \eqref{p3-2.43}, \eqref{p3-3.10}, the regularity of $v^5$ and $v^6$, and the definition of $a(x)$, system \eqref{p3-f1ker}-\eqref{p3-f6ker} implies

\begin{eqnarray}
\left(\la^2 -\frac{l^2k_3}{\rho_1}\right)v^1 +\frac{k_1}{\rho_1}(v^1_x+v^3)_x =0 \ \ \text{in} \ \ (0,\beta),\label{p3-2.45}\\ 
\left(\la^2 -\frac{k_1}{\rho_2}\right)v^3+\frac{k_2}{\rho_2}v^3_{xx}-\frac{k_1}{\rho_2}v^1_x=0 \ \ \text{in} \ \ (0,\beta),\label{p3-2.46}\\
\frac{k_1}{\rho_1}(v^1_x+v^3)=-\frac{k_3}{\rho_1}v^1_x, \ \ \text{in} \ \ (0,\beta).\label{p3-2.47}
\end{eqnarray}
Inserting \eqref{p3-2.47} in  \eqref{p3-2.45}, we obtain
\begin{equation}\label{p3-2.48}
sv^1 +v^1_{xx}=0 \ \ \text{in} \ \ (0,\beta),
\end{equation}
where $s=\dfrac{l^2k_3-\rho_1\la^2}{k_3}$. Let us introduce the following three cases.\\\linebreak
\textbf{Case 1:} If $\displaystyle \la^2=\frac{l^2k_3}{\rho_1}$. Then, from \eqref{p3-2.48}, we deduce that
\begin{equation}
	v^1(x)=c_1x+c_2 \ \ \text{in} \ \ (0,\beta), \quad   c_1,c_2 \in \C.
\end{equation}
Using the fact that $v^1(0)=0$, we get 
\begin{equation}\label{p3-2.49}
	c_2 =0 \ \ \text{and consequently} \ \ v^1(x)=c_1 x  \ \ \text{in} \ \ (0,\beta).
\end{equation}
Inserting \eqref{p3-2.49} in \eqref{p3-2.47}, we get 
\begin{equation}\label{p3-2.50}
	v^3(x)=-\left(1+\frac{k_3}{k_1}\right)c_1 \ \ \text{in} \ \ (0,\beta).
\end{equation}
Now, from \eqref{p3-2.49}, \eqref{p3-2.50} and the fact that $v^3(0)=0$, we get
\begin{equation}\label{p3-2.52}
	c_1=0, \ \ v^1=0 \ \text{in}\ (0,\beta) \ \ \text{and}  \ \ v^3=0 \ \text{in}\ (0,\beta).
\end{equation}
Thus, from \eqref{p3-f1ker}, \eqref{p3-f3ker}, \eqref{p3-2.43}, \eqref{p3-3.10}, \eqref{p3-2.52} and the fact that $\la \neq 0$, we obtain 
\begin{equation}\label{p3-3.19}
	U=0 \ \ \text{in} \ \ (0,\beta).
\end{equation}
Let $V=(v^1,v^1_x,v^3,v^3_x,v^5,v^5_x)^{\top}$. From \eqref{p3-2.52} and the regularity of $v^i $, $i\in \{1,3,5\}$ , we get $V(\beta)=0$. Now, by inserting \eqref{p3-f1ker}, \eqref{p3-f3ker} and \eqref{p3-f5ker} in \eqref{p3-f2ker}, \eqref{p3-f4ker} and \eqref{p3-f6ker} respectively, then system \eqref{p3-f1ker}-\eqref{p3-f6ker} can be written in $(\beta,L)$  as the following
\begin{equation}\label{p3-de}
V_x =A V \ \ \text{in}\ \ (\beta,L),
\end{equation}where
\begin{equation*}
A =  \begin{pmatrix}
0&1&0&0&0&0\\
0&0&0&1&0&l(1-\frac{k_3}{k_1})\\
0&0&0&1&0&0\\
0&-\frac{k_1}{k_2}&\frac{\rho_2\la^2-k_1}{k_2}&0&-\frac{lk_1}{k_2}&0\\
0&0&0&0&0&1\\
0&-l(\frac{k_1}{k_3}+1)&0&-l\frac{k_1}{k_3}&\frac{\rho_1\la^2-l^2k_1}{k_3}&0
\end{pmatrix}.
\end{equation*}The solution of the differential equation \eqref{p3-de} is given by
\begin{equation}\label{p3-solde}
V(x)=e^{A (x-\beta )}V (\beta ).
\end{equation}
Thus, from \eqref{p3-solde} and the fact that $V(\beta )=0$, we get 
\begin{equation}\label{p3-2.55}
V=0\ \ \text{in}\ \ (\beta ,L) \ \ \text{and consequently} \ U=0 \ \ \text{in}\ \ (\beta,L).
\end{equation}
Therefore, from \eqref{p3-3.19} and \eqref{p3-2.55}, we obtain 
$$
U=0 \ \ \text{in} \ \ (0,L).
$$
\textbf{Case 2:} If $\displaystyle \la^2 > \frac{l^2k_3}{\rho_1}$. Then, from \eqref{p3-2.48}, we deduce that 
\begin{equation}\label{p3-2.56}
	v^1(x)=c_1e^{\sqrt{-s}x}+c_2e^{-\sqrt{-s}x} \ \ \text{in} \ \ (0,\beta), \ \ c_1, c_2 \in \C.
\end{equation}
 Now, from \eqref{p3-2.56} and the fact that $v^1(0)=0$, we get 
\begin{equation}\label{p3-2.57}
c_2=-c_1 \ \ \text{and consequaently} \ \ v^1(x)=c_1(e^{\sqrt{-s}x}-e^{-\sqrt{-s}x}) \ \ \text{in} \ \ (0,\beta).
\end{equation}
Inserting \eqref{p3-2.57} in \eqref{p3-2.47}, we get 
\begin{equation}\label{p3-2.58}
	v^3(x) =-\left(1+\frac{k_3}{k_1}\right)\sqrt{-s}\left(e^{\sqrt{-s}x}+e^{-\sqrt{-s}x}\right)c_1 \ \ \text{in} \ \ (0,\beta).
\end{equation}
From \eqref{p3-2.57}, \eqref{p3-2.58} and the fact that $v^3(0)=0$, we obtain
\begin{equation}
c_1=0, \ \ v^1=0 \ \ \text{in} \ \ (0,\beta),\ \ v^3=0  \ \ \text{in} \ \ (0,\beta) \ \ \text{and consequently} \ \ U=0 \ \ \text{in} \ \ (0,\beta).
\end{equation}
Similarly to Case 1, we get $U=0$ in $(\beta,L)$ and consequently $U=0$ in $(0,L)$.\\\linebreak
\textbf{Case 3:} If $\displaystyle \la^2 < \frac{l^2k_3}{\rho_1}$. Then, from \eqref{p3-2.48}, we deduce that 
\begin{equation}\label{p3-2.60}
v^1(x)=c_1\cos(\sqrt{s}x)+c_2\sin(\sqrt{s}x) \ \ \text{in} \ \ (0,\beta), \ \ c_1, c_2 \in \C.
\end{equation}
 Now, from \eqref{p3-2.60} and the fact that $v^1(0)=0$, we get 
\begin{equation}\label{p3-2.61}
c_1=0 \ \ \text{and consequaently} \ \ v^1(x)=c_2\sin(\sqrt{s}x) \ \ \text{in} \ \ (0,\beta).
\end{equation}
Inserting \eqref{p3-2.61} in \eqref{p3-2.47}, we get 
\begin{equation}\label{p3-2.62}
v^3 (x)=-\left(1+\frac{k_3}{k_1}\right)\sqrt{s}\cos(\sqrt{s}x)c_2 \ \ \text{in} \ \ (0,\beta).
\end{equation}
From \eqref{p3-2.62}, \eqref{p3-2.61} and the fact that $v^3(0)=0$, we obtain
\begin{equation}
c_2=0, \ \ v^1=0 \ \ \text{in} \ \ (0,\beta),\ \ v^3=0 \ \ \text{in} \ \ (0,\beta) \ \ \text{and consequently} \ \ U=0 \ \ \text{in} \ \ (0,\beta).
\end{equation}
Similarly to Case 1, we get $U=0$ in $(\beta,L)$ and consequently $U=0$ in $(0,L)$. The proof is thus complete.
\end{proof}

\section{Exponential and Polynomial Stability }\label{p3-secpoly}
\noindent In this section, we show the influence of the physical coefficients on the stability of system \eqref{p3-sysorig}-\eqref{p3-initialcond}. The main results of this section are the following theorems. 
\begin{theoreme}\label{exps}{\rm
	 If
	 \begin{equation*}
	 	\frac{k_1}{\rho_1}=\frac{k_2}{\rho_2} \ \ \text{and} \ \ k_1=k_3,
	 \end{equation*} 
	 then the $C_0-$semigroup $e^{t\AA}$ is exponentially stable; i.e. there exists constants $M\geq 1$ and $\epsilon>0$ independent of $U_{0}$ such that 
	\begin{equation}\label{1}
		\|e^{t\AA}U_{0}\|_{\HH}\leq Me^{-\epsilon t}\|U_{0}\|_{\HH}.
	\end{equation}}
\end{theoreme}
\begin{theoreme}\label{p3-pol-eq}{\rm
		 If \begin{equation*}
	\displaystyle \frac{k_1}{\rho_1}=\frac{k_2}{\rho_2} \ \ \text{and} \ \ k_1\neq k_3,
		 \end{equation*} 
		 then there exists $C>0$ such that for every $U_{0}\in D(\AA)$, we have 
	\begin{equation}\label{2}
		E(t)\leq \frac{C}{t}\|U_0\|^2_{D(\AA)},\quad t>0.
		\end{equation}}
\end{theoreme}
\begin{theoreme}\label{p3-pol-neq}{\rm
		If \begin{equation*}
		\displaystyle \frac{k_1}{\rho_1}\neq \frac{k_2}{\rho_2},	
		\end{equation*}
		 then there exists $C>0$ such that for every $U_{0}\in D(\AA)$, we have 
	\begin{equation}\label{3}
		E(t)\leq \frac{C}{\sqrt{t}}\|U_0\|^2_{D(\AA)},\quad t>0.
		\end{equation}}
\end{theoreme}
\noindent 
According to \cite{Huang01}, \cite{pruss01} and Theorem 2.4 in \cite{Borichev01} (see also \cite{Batty01} and \cite{RaoLiu01}), a $C_0-$semigroup of contractions $\left(e^{t\AA}\right)_{t\geq0}$ on $\HH$ satisfy \eqref{1}, \eqref{2} and \eqref{3} if 
\begin{equation}\tag{${\rm M_1}$}\label{p3-M1-cond}
	i\R \subset \rho(\AA)
\end{equation}
\begin{equation}\tag{${\rm M_2}$}\label{p3-H-cond}
\sup_{\la\in \R}\left\|\left(i\la I-\AA\right)^{-1}\right\|_{\mathcal{L}(\HH)}=O\left(\abs{\la}^{\ell}\right), \quad \text{with} \quad \left\{\begin{array}{lll}
\ell=0 \ \ \text{for Theorem \ref{exps}},\vspace{0.15cm}\\
\ell=2 \ \ \text{for Theorem \ref{p3-pol-eq}},\vspace{0.15cm}\\
\ell=4 \ \ \text{for Theorem \ref{p3-pol-neq}}.
\end{array}\right.
\end{equation}
Since $i\R\subset \rho(\AA)$ (see Section \ref{p3-sec3}), then condition \eqref{p3-M1-cond} is satisfied. We will prove condition \eqref{p3-H-cond} by a contradiction argument. For this purpose,
suppose that \eqref{p3-H-cond} is false, then there exists $\left\{(\la^n,U^n)\right\}_{n\geq1}\subset \R^{\ast} \times D(\mathcal{A})$ with
\begin{equation}\label{p3-contra-pol2}
|\la^n|\to\infty \quad \text{and}\quad \|U^n\|_{\mathcal{H}}=\|(v^{1,n},v^{2,n},v^{3,n},v^{4,n},v^{5,n},v^{6,n})^{\top}\|_{\HH}=1,
\end{equation}
such that  
\begin{equation}\label{p3-eq0ps}
(\la^n )^{\ell} (i\la^n I-\AA )U^n =F^n:=(f^{1,n},f^{2,n},f^{3,n},f^{4,n},f^{5,n},f^{6,n})^{\top}  \to 0  \quad \text{in}\quad \HH.
\end{equation} 
For simplicity, we drop the index $n$. Equivalently, from \eqref{p3-eq0ps}, we have

\begin{eqnarray}
	i\la v^1 -v^2 &=&\la^{-\ell}f^1,\label{p3-f1ps}\\
	i\la \rho_1 v^2 -k_1 (v^1_x +v^3+lv^5)_x -lk_3 (v^5_x-lv^1)&=&\rho_1 \la^{-\ell}f^2,\label{p3-f2ps}\\
	i\la v^3-v^4&=&\la^{-\ell}f^3,\label{p3-f3ps}\\
	i\la \rho_2 v^4 -k_2 v^3_{xx}+k_1 (v^1_x+v^3+lv^5)&=&\rho_2 \la^{-\ell}f^4,\label{p3-f4ps}\\
	i\la v^5-v^6&=&\la^{-\ell}f^5,\label{p3-f5ps}\\
	i\la \rho_1 v^6 -k_3 (v^5_x-lv^1)_x +lk_1 (v^1_x+v^3+lv^5)+a(x)v^6&=&\rho_1 \la^{-\ell}f^6. \label{p3-f6ps}
\end{eqnarray}
By inserting \eqref{p3-f1ps} in \eqref{p3-f2ps}, \eqref{p3-f3ps} in \eqref{p3-f4ps} and \eqref{p3-f5ps} in \eqref{p3-f6ps}, we deduce that
\begin{eqnarray}
\la^2 \rho_1 v^1 +k_1 (v^1_x +v^3+lv^5)_x +lk_3 (v^5_x-lv^1)&=&-\rho_1 \la^{-\ell}f^2-i\rho_1 \la^{-\ell+1}f^1,\label{p4-4.18*}\\
\la^2 \rho_2 v^3 +k_2 v^3_{xx} -k_1(v^1_x+v^3+lv^5)&=&-\rho_2 \la^{-\ell}f^4-i\rho_2 \la^{-\ell+1}f^3,\label{p4-4.18**}\\
\la^2\rho_1 v^5 +k_3 (v^5_x-lv^1)_x -lk_1 (v^1_x+v^3+lv^5)-a(x)v^6&=&-\rho_1 \la^{-\ell}f^6 -i\rho_1\la^{-\ell+1}  f^5.\label{p4-4.18***}
\end{eqnarray}
Here we will check the condition \eqref{p3-H-cond} by finding a contradiction with \eqref{p3-contra-pol2} by showing $\left\|U\right\|_{\HH}=o(1)$.  For clarity, we divide the proof into several Lemmas. From the above system and the fact that $\ell \in \{2,4\}$, $\|U\|_\HH=1$ and $\|F\|_\HH=o(1)$, we remark that 
$$
\begin{array}{lll}
\displaystyle \|v^1\|=O\left(\left|\la\right|^{-1}\right), \  \|v^3\|=O\left(\left|\la\right|^{-1}\right), \  \|v^5\|=O\left(\left|\la\right|^{-1}\right), \  \|v^1_{xx}\|=O\left(\left|\la\right|\right), \  \|v^3_{xx}\|=O\left(\left|\la\right|\right)\ \text{and} \  \left\|v^5_{xx} \right\|=O\left(\left|\la\right|\right).
%\displaystyle \left\|v^5_{xx} \right\|=O\left(\left|\la\right|\right).
\end{array}
$$
Also, from Poincar\'e inequality and the fact that $\|F\|_\HH=o(1)$, we remark that 
\begin{equation*}\label{p3-poincare}
\|f^1\|=o(1), \ \ 	\|f^3\|=o(1) \ \ \text{and} \ \ 	\|f^5\|=o(1).
\end{equation*}
\\
We define the following hypotheses:
	 \begin{equation}\tag{{\rm H$_1$}}\label{H1}
\frac{k_1}{\rho_1}=\frac{k_2}{\rho_2}, \ \ k_1=k_3 \ \ \text{and} \ \ \ell =0;
\end{equation} 
\begin{equation}\tag{{\rm H$_2$}}\label{H2}
 \displaystyle \frac{k_1}{\rho_1}=\frac{k_2}{\rho_2}, \ \ k_1\neq k_3\ \ \text{and} \ \ \ell=2;
\end{equation}
\begin{equation}\tag{{\rm H$_3$}}\label{H3}
		 \displaystyle \frac{k_1}{\rho_1}\neq \frac{k_2}{\rho_2} \ \ \text{and} \ \ \ell=4;
\end{equation}  
\begin{rk}
	{\rm 
		According to Remark 3.8 in \cite{Wehbe03}, the case of equal speed propagation (i.e., when \eqref{H1} holds) has only mathematical sound. $\hfill\square$
		
	}
\end{rk}
\begin{lem}\label{p4-1stlemps}
		{\rm If \eqref{H1} or \eqref{H2} or \eqref{H3} holds. then the solution $U=(v^1,v^2,v^3,v^4,v^5,v^6)^{\top}\in D(\AA)$ of  \eqref{p3-f1ps}-\eqref{p3-f6ps} satisfies the following estimations 
		\begin{equation}\label{p3-v65xx}
			\int_{0}^{\beta } \left|v^6\right|^2dx=o(\la^{-\ell}) \ \ \text{and} \ \  \int_{0}^{\beta } \left|v^5\right|^2dx=o(\la^{-\ell-2}).
		\end{equation}
	}
\end{lem}
\begin{proof}
		First, taking the inner product of \eqref{p3-eq0ps} with $U$ in $\HH$ and using \eqref{p3-Reau}, we get
	\begin{equation}\label{p3-4.10}
	\displaystyle	\intdx a(x)\left|v^6\right|^2dx =a_0 \intdnb \left|v^6\right|^2dx=-\Re \left(\AA U,U\right)_{\HH}=\la^{-\ell}\Re  \left(F,U\right)_{\HH} \leq \la^{-\ell} \|F\|_{\HH}\|U\|_{\HH} .
	\end{equation}Thus, from \eqref{p3-4.10}  and the fact that $\|F\|_{\HH}=o(1)$ and $\|U\|_{\HH}=1$, we obtain the first estimation in \eqref{p3-v65xx}. From  \eqref{p3-f5ps}, we deduce that 
	\begin{equation}\label{p3-4.11}
		\intdnb \left|v^5\right|^2dx \leq \frac{1}{\la^{2}}	\intdnb \left|v^6\right|^2dx+\frac{1}{\la^{2\ell+2}}	\intdnb \left|f^5\right|^2dx.
	\end{equation}
Finally, from \eqref{p3-4.11}, the first estimation in \eqref{p3-v65xx}, and the fact that $\ell\in\{0,2,4\}$, $\|f^5\|=o(1)$, we get the second estimation in \eqref{p3-v65xx}. The proof is thus complete.
\end{proof}\\\linebreak
For all $\displaystyle 0<\varepsilon<\frac{\beta}{12}$, we fix the following cut-off functions \\ \linebreak
\begin{itemize}
\item  $\f_j\in C^{2}\left([0,L]\right)$, $j\in \{1,\cdots,6\}$ such that $0\leq \f_j (x)\leq 1$, for all $x\in[0,L]$ and 
\begin{equation*}
\f_j (x)= 	\left \{ \begin{array}{lll}
1 &\text{if} \quad \,\,  x \in [j\varepsilon ,\beta -j\varepsilon],&\vspace{0.1cm}\\
0 &\text{if } \quad x \in [0,(j-1)\varepsilon]\cup [\beta +(1-j)\varepsilon ,L].&
\end{array}	\right. \qquad\qquad
\end{equation*}\\
\item  $\q_1, \q_2 \in C^{1}\left([0,L]\right)$ such that $0\leq \q_1 (x)\leq 1$, \  $0\leq \q_2 (x)\leq 1$ for all $x\in[0,L]$ and 
\begin{equation*}
\q_1 (x)= 	\left \{ \begin{array}{lll}
1 &\text{if} \quad \,\,  x \in [0,\alpha_1] ,&\vspace{0.1cm}\\
0 &\text{if } \quad x \in [\alpha_2,L],&
\end{array}	\right. 
\text{and} \quad 
\q_2 (x)= 	\left \{ \begin{array}{lll}
0 &\text{if} \quad \,\,  x \in [0,\alpha_1] ,&\vspace{0.1cm}\\
1 &\text{if } \quad x \in [\alpha_2,L ],&
\end{array}	\right. \text{with} \ \ 0<\alpha_1 <\alpha_2 <\beta <L.
\end{equation*}
\\ 
\end{itemize}

\begin{lem}\label{p4-lem2}
{\rm If \eqref{H1} or \eqref{H2} or \eqref{H3} holds. then the solution $U=(v^1,v^2,v^3,v^4,v^5,v^6)^{\top}\in D(\AA)$ of  \eqref{p3-f1ps}-\eqref{p3-f6ps} satisfies the following estimations 
	\begin{equation}\label{p4-4.15}
\int_{\varepsilon}^{\beta-\varepsilon}|v^5_x|^2 dx =\frac{o(1)}{|\la|^{\min\left(\frac{\ell}{2}+1,\ell \right)}}=
\left\{\begin{array}{lll}
\displaystyle \frac{o(1)}{|\la|^\ell} & \text{if} & \ell \in \{0,2\},\vspace{0.25cm}\\
\displaystyle \frac{o(1)}{|\la|^{\frac{\ell}{2}+1}} & \text{if} & \ell \in \{2,4\}.
\end{array}\right.
	\end{equation}
	\begin{proof}
	First, multiplying \eqref{p3-f6ps} by $\f_1 \overline{v^5}$, integrating over $(0,\beta)$, and using the fact that $\|v^5\|=O\left(|\la|^{-1}\right)$ , $\|f^6\|=o(1)$, we obtain
	\begin{equation*}
	\begin{array}{lll}
	\displaystyle	i\la  \rho_1 \intdnb \f_1 v^6 \overline{v^5}dx -k_3 \intdnb \f_1v^5_{xx} \overline{v^5}dx +lk_3\intdnb \f_1 v^1_x \overline{v^5}dx +lk_1 \intdnb \f_1 (v^1_x+v^3+lv^5)\overline{v^5}dx\vspace{0.25cm}\\
	\displaystyle +\, a_0\intdnb \f_1 v^6\overline{v^5}dx=o(\la^{-\ell}).
	\end{array}
	\end{equation*}
	Using integration by parts in the above equation and the fact that $\f_1 (0)=\f_1 (\beta)=0$, we get 
		\begin{equation}\label{p4-4.16}
	\begin{array}{lll}
	\displaystyle k_3\intdnb \f_1|v^5_x|^2 dx =-\,k_3\intdnb \f_1^\prime v^5_x \overline{v^5}dx -	i\la  \rho_1 \intdnb \f_1 v^6 \overline{v^5}dx  -lk_3\intdnb \f_1 v^1_x \overline{v^5}dx \vspace{0.25cm}\\
\hspace{3cm}	\displaystyle -\,lk_1 \intdnb \f_1 (v^1_x+v^3+lv^5)\overline{v^5}dx- a_0 \intdnb \f_1 v^6\overline{v^5}dx+o(\la^{-\ell}).
	\end{array}
	\end{equation}
	Using the above estimation, Lemma \ref{p4-1stlemps} and the fact that $v^1_x$, $v^5_x$, $(v^1_x+v^3+lv^5)$ are uniformly bounded in $L^2(0,L)$, $\ell \in \{0,2,4\}$, we obtain
	$$
	k_3\intdnb \f_1 |v^5_x|^2 dx =\frac{o(1)}{|\la|^{\min\left(\frac{\ell}{2}+1,\ell \right)}}.
	$$
Finally, from the above estimation and the definition of $\f_1$, we obtain \eqref{p4-4.15}. The proof is thus complete.
	\end{proof}
}
\end{lem}
\begin{lem}\label{p4-lem3}
	{\rm If \eqref{H2} or \eqref{H3} holds. then the solution $U=(v^1,v^2,v^3,v^4,v^5,v^6)^{\top}\in D(\AA)$ of  \eqref{p3-f1ps}-\eqref{p3-f6ps} satisfies the following estimations 
	\begin{equation}\label{p4-4.12}
	\int_{2\varepsilon}^{\beta-2\varepsilon}\left|v^1_x\right|^2 dx =o(1) \quad \text{and} \quad \int_{3\varepsilon}^{\beta-3\varepsilon}\left|\la v^1\right|^2 dx =o(1).
	\end{equation}	
		
	}
\end{lem}
\begin{proof}
First, multiplying \eqref{p3-f6ps} by $\mathsf{f}_2\overline{v^1_x}$, integrating over $(\varepsilon,\beta-\varepsilon)$,  and using the fact that $v^1_x$ is uniformly bounded in $L^2(0,L)$ and $\|f^6\|=o(1)$, we get 
	\begin{equation*}
	\begin{array}{lll}
	\displaystyle l(k_1+k_3) \int_{\varepsilon}^{\beta-\varepsilon}\f_2 |v^1_x|^2 dx=-	i\la \rho_1 \int_{\varepsilon}^{\beta-\varepsilon} \f_2 v^6\overline{v^1_x}dx +k_3\int_{\varepsilon}^{\beta-\varepsilon} \f_2 v^5_{xx}\overline{v^1_x }dx- lk_1 \int_{\varepsilon}^{\beta-\varepsilon}\f_2 (v^3+lv^5)\overline{v^1_x}dx\vspace{0.25cm}\\
	\hspace{4.5cm} \displaystyle -\, a_0 \int_{\varepsilon}^{\beta-\varepsilon}\f_2 v^6 \overline{v^1_x}dx+o(|\la|^{-\ell}),
			\end{array}
	\end{equation*}
%	$$
%	\left|	i\la \rho_1 \intdx \h v^6\overline{v^1_x}dx\right|=	\left|	i\la \rho_1 \intdnb \h v^6\overline{v^1_x}dx\right|\leq |\la |\rho_1 \left(\intdnb \left|v^6\right|^2dx\right)^{\frac{1}{2}}\left\|v^1_x \right\|=o(|\la|^{-1})
%	$$
using integration by parts and the fact that $\f_2 (\varepsilon)=\f_2 (\beta-\varepsilon)=0$, we get 
\begin{equation*}
\begin{array}{lll}
\displaystyle l(k_1+k_3) \int_{\varepsilon}^{\beta-\varepsilon}\f_2 |v^1_x|^2 dx=-	i\la \rho_1 \int_{\varepsilon}^{\beta-\varepsilon} \f_2 v^6\overline{v^1_x}dx +k_3\int_{\varepsilon}^{\beta-\varepsilon} \f_2 v^5_{x}\overline{v^1_{xx} }dx +k_3\int_{\varepsilon}^{\beta-\varepsilon} \f_2^\prime v^5_{x}\overline{v^1_{x} }dx\vspace{0.25cm}\\
\hspace{4.5cm} \displaystyle -\, lk_1 \int_{\varepsilon}^{\beta-\varepsilon}\f_2 (v^3+lv^5)\overline{v^1_x}dx- a_0 \int_{\varepsilon}^{\beta-\varepsilon}\f_2 v^6 \overline{v^1_x}dx+o(|\la|^{-\ell}).
\end{array}
\end{equation*}
Using the above equation, Lemmas \ref{p4-1stlemps}-\ref{p4-lem2} with $\ell\in\{2,4\}$, and the fact that $v^1_x$ is uniformly bounded in $L^2(0,L)$, $\|v^1_{xx}\|=O(|\la|)$, $\|v^3\|=O(|\la|^{-1})$, we obtain
$$
l(k_1+k_3) \int_{\varepsilon}^{\beta-\varepsilon}\f_2 |v^1_x|^2 dx=\frac{o(1)}{|\la|^{\frac{\ell}{4}-\frac{1}{2}}}.
$$
Thus, from the above estimation, the definition of $\f_2$ and the fact that $\frac{\ell}{4}-\frac{1}{2}\in\left\{0,\frac{1}{2} \right\}$, we obtain the first estimation in \eqref{p4-4.12}. Now, 
Multiplying \eqref{p4-4.18*} by $\f_3\overline{v^1}$, integrating over $(2\varepsilon,\beta-2\varepsilon)$, using integration by parts and the definition of $\mathsf{f}_3$, then using the fact that $\|v^1\|=O(|\la|^{-1})$, $\|f^1\|=o(1)$ and $\|f^2\|=o(1)$, we get 
\begin{equation}\label{p4-4.18}
\begin{array}{lll}
\displaystyle \rho_1 \int_{2\varepsilon}^{\beta-2\varepsilon}\f_3\left|\la v^1\right|^2 dx =k_1 \int_{2\varepsilon}^{\beta-2\varepsilon}\f_3^{\prime}(v^1_x+v^3+lv^5)\overline{v^1}dx +k_1 \int_{2\varepsilon}^{\beta-2\varepsilon}\f_3 \left|v^1_x\right|^2dx \vspace{0.25cm}\\ \displaystyle  +\,k_1 \int_{2\varepsilon}^{\beta-2\varepsilon}\f_3 (v^3+lv^5)\overline{v^1_x }dx-lk_3 \int_{2\varepsilon}^{\beta-2\varepsilon}\f_3 v^5_x \overline{v^1}dx +l^2 k_3 \int_{2\varepsilon}^{\beta-2\varepsilon}\f_3 \left|v^1\right|^2dx+o(|\la|^{-\ell+1}).
\end{array}
\end{equation}
From \eqref{p4-4.18}, Lemma \ref{p4-lem2}, the first estimation in \eqref{p4-4.12}, and the fact that $\|v^1\|=O(|\la|^{-1})$, $\|v^3\|=O(|\la|^{-1})$, $\|v^5\|=O(|\la|^{-1})$ and $\ell \in \{2,4\}$, we obtain
\begin{equation}
	\rho_1 \int_{2\varepsilon}^{\beta-2\varepsilon}\f_3\left|\la v^1\right|^2 dx=o(1).
\end{equation}
Finally, from the above estimation and the definition of $\f_3$, we obtain the second estimation desired. The proof is thus complete.
\end{proof}
\begin{lem}\label{p4-lem4}
	{\rm    
	If \eqref{H1}  holds, then the solution $U=(v^1,v^2,v^3,v^4,v^5,v^6)^{\top}\in D(\AA)$ of  \eqref{p3-f1ps}-\eqref{p3-f6ps} satisfies the following estimations
		\begin{equation}\label{p4-4.19}
		\int_{2\varepsilon}^{\beta-2\varepsilon} |v^1_x|^2 dx =o(1) \ \ \text{and} \ \ \int_{2\varepsilon}^{\beta-2\varepsilon}|\la v^1|^2 dx =o(1).
		\end{equation} 
	}
\end{lem}
		\begin{proof}
			First, take $\ell=0$ in \eqref{p4-4.18*} and multiply it
			 by  $\f_2 (\overline{v^5_x}-l\overline{v^1})$, integrating over $(\varepsilon,\beta-\varepsilon)$, and taking the real part, we get
			\begin{equation*}
			\begin{array}{lll}
	\displaystyle	\Re \left\{	\la^2 \rho_1 \int_{\varepsilon}^{\beta-\varepsilon}\f_2 v^1 (\overline{v^5_x}-l\overline{v^1})dx +k_1\int_{\varepsilon}^{\beta-\varepsilon}\f_2 v^1_{xx}(\overline{v^5_x}-l\overline{v^1})dx+k_1 \int_{\varepsilon}^{\beta-\varepsilon}\f_2 (v^3_x+lv^5_x)(\overline{v^5_x}-l\overline{v^1})dx\right.\vspace{0.25cm}\\
	\hspace{1cm}	\displaystyle \left. +\,lk_3\int_{\varepsilon}^{\beta-\varepsilon}\f_2 |v^5_x-lv^1|^2 dx \right\}=\underbrace{\Re\left\{-\rho_1 \int_{\varepsilon}^{\beta-\varepsilon}\f_2 f^2(\overline{v^5_x}-l\overline{v^1})dx-i\la\rho_1 \int_{\varepsilon}^{\beta-\varepsilon}\f_2 f^1(\overline{v^5_x}-l\overline{v^1})dx\right\}}_{:=\mathtt{I}_1},
		\end{array}
			\end{equation*}
			using integration by parts and the fact that $\f_2 (\varepsilon)=\f_2(\beta-\varepsilon)=0$, then using Lemmas \ref{p4-1stlemps}-\ref{p4-lem2} with $\ell =0$ and the fact that $\|v^1\|=O(|\la|^{-1})$, $\|f^1\|=o(1)$, $\|f^1_x\|=o(1)$, we get
			\begin{equation*}
			\mathtt{I}_1 = \Re\left\{-\rho_1 \int_{\varepsilon}^{\beta-\varepsilon}\f_2 f^2(\overline{v^5_x}-l\overline{v^1})dx+i\la\rho_1 \int_{\varepsilon}^{\beta-\varepsilon}\f_2 f^1_x\overline{v^5}dx +i\la\rho_1\int_{\varepsilon}^{\beta-\varepsilon}\f_2^\prime  f^1\overline{v^5}dx+i\la l\rho_1\int_{\varepsilon}^{\beta-\varepsilon}\f_2 f^1 \overline{v^1}dx\right\}=o(1),
			\end{equation*}
			consequently, we obtain
				\begin{equation}\label{p4-4.20}
			\begin{array}{lll}
			\displaystyle	\Re \left\{	\la^2 \rho_1 \int_{\varepsilon}^{\beta-\varepsilon}\f_2 v^1 \overline{v^5_x}dx-l\rho_1 \int_{\varepsilon}^{\beta-\varepsilon}\f_2 |\la v^1|^2 dx +k_1\int_{\varepsilon}^{\beta-\varepsilon}\f_2 v^1_{xx}\overline{v^5_x}dx-lk_1 \int_{\varepsilon
			}^{\beta-\varepsilon}\f_2 v^1_{xx}\overline{v^1}dx\right.\vspace{0.25cm}\\
			\hspace{1cm}	\displaystyle \left.+\,k_1 \int_{\varepsilon}^{\beta-\varepsilon}\f_2 (v^3_x+lv^5_x)(\overline{v^5_x}-l\overline{v^1})dx +lk_3 \int_{\varepsilon}^{\beta-\varepsilon}\f_2 |v^5_x-lv^1|dx \right\}=o(1).
			\end{array}
			\end{equation}
			Now, multiplying \eqref{p3-f3ps} by $\f_2 \overline{v^1_x}$, 
integrating over $(\varepsilon,\beta-\varepsilon)$, taking the real part, we get
			\begin{equation*}\label{p4-4.21}
				\begin{array}{lll}
				\displaystyle \Re \left\{\la^2 \rho_1 \int_{\varepsilon}^{\beta-\varepsilon}\f_2 v^5\overline{v^1_x}dx+k_3 \int_{\varepsilon}^{\beta-\varepsilon}\f_2 v^5_{xx}\overline{v^1_x}dx -l(k_1+k_3)\int_{\varepsilon}^{\beta-\varepsilon}\f_2 |v^1_x|^2 dx -lk_1\int_{\varepsilon}^{\beta-\varepsilon}\f_2 (v^3+lv^5)\overline{v^1_x}dx\right.\vspace{0.25cm}\\
			\hspace{2cm}	\displaystyle \left. -\, a_0 \int_{\varepsilon}^{\beta-\varepsilon}\f_2 v^6\overline{v^1_x}dx\right\}= \underbrace{\Re\left\{-\rho_1 \int_{\varepsilon}^{\beta-\varepsilon}\f_2 f^6\overline{v^1_x}dx-i\la\rho_1 \int_{\varepsilon}^{\beta-\varepsilon}\f_2 f^5\overline{v^1_x}dx\right\}}_{:=\mathtt{I}_2},
				\end{array}
			\end{equation*}
				using integration by parts and the fact that $\f_2 (\varepsilon)=\f_2(\beta-\varepsilon)=0$, then using the fact that $v^1_x$ is uniformly bounded in $L^2(0,L)$, $\|v^1\|=O(|\la|^{-1})$, $\|f^5\|=o(1)$, $\|f^5_x\|=o(1)$, $\|f^6\|=o(1)$, we get
			\begin{equation*}
			\mathtt{I}_2 = \Re\left\{-\rho_1 \int_{\varepsilon}^{\beta-\varepsilon}\f_2 f^6\overline{v^1_x}dx+i\la\rho_1 \int_{\varepsilon}^{\beta-\varepsilon}\f_2 f^5_x\overline{v^1}dx +i\la\rho_1\int_{\varepsilon}^{\beta-\varepsilon}\f_2^\prime  f^5\overline{v^1}dx\right\}=o(1),
			\end{equation*}
			consequently, by using integration by parts in \eqref{p4-4.21} and the fact that $\f_2 (\varepsilon)=\f_2(\beta-\varepsilon)=0$, we obtain 
				\begin{equation}\label{p4-4.22}
			\begin{array}{lll}
			\displaystyle \Re \left\{-\la^2 \rho_1 \int_{\varepsilon}^{\beta-\varepsilon}\f_2 v^5_x\overline{v^1}dx-\la^2 \rho_1 \int_{\varepsilon}^{\beta-\varepsilon}\f_2^\prime v^5\overline{v^1}dx-k_3 \int_{\varepsilon}^{\beta-\varepsilon}\f_2 v^5_{x}\overline{v^1_{xx}}dx-k_3\int_{\varepsilon}^{\beta-\varepsilon}\f_2^\prime v^5_{x}\overline{v^1_x}dx  \right.\vspace{0.25cm}\\
			\hspace{1cm}	\displaystyle \left. -\,l(k_1+k_3)\int_{\varepsilon}^{\beta-\varepsilon}\f_2 |v^1_x|^2 dx - lk_1 \int_{\varepsilon}^{\beta-\varepsilon}\f_2 (v^3+lv^5)\overline{v^1_x}dx- a_0 \int_{\varepsilon}^{\beta-\varepsilon}\f_2 v^6\overline{v^1_x}dx\right\}=o(1).
			\end{array}
			\end{equation}
			Adding \eqref{p4-4.20} and \eqref{p4-4.22} and using the fact that $k_1=k_3$, we get 
					\begin{equation*}\label{p4-4.24}
			\begin{array}{lll}
 		\displaystyle l\rho_1 \int_{\varepsilon}^{\beta-\varepsilon}\f_2 |\la v^1|^2 dx +2lk_1  \int_{\varepsilon}^{\beta-\varepsilon}\f_2|v^1_x|^2 dx +\underbrace{lk_1 \int_{\varepsilon}^{\beta-\varepsilon}\f_2v^1_{xx}\overline{v^1}dx}_{:=\mathtt{I}_3}=\Re\left\{ lk_3  \int_{\varepsilon}^{\beta-\varepsilon}\f_2 |v^5_x-lv^1|dx\right.\vspace{0.25cm}\\
	\displaystyle 	 \left.+\, k_1  \int_{\varepsilon}^{\beta-\varepsilon}\f_2 (v^3_x+lv^5_x)(\overline{v^5_x}-l\overline{v^1})dx -\la^2 \rho_1  \int_{\varepsilon}^{\beta-\varepsilon}\f_2^\prime v^5\overline{v^1}dx-k_3  \int_{\varepsilon}^{\beta-\varepsilon}\f_2^\prime v^5_{x}\overline{v^1_x}dx\right.\vspace{0.25cm}\\
\left.	\displaystyle -\, lk_1  \int_{\varepsilon}^{\beta-\varepsilon}\f_2 (v^3+lv^5)\overline{v^1_x}dx- a_0  \int_{\varepsilon}^{\beta-\varepsilon}\f_2 v^6\overline{v^1_x}dx\right\}+o(1),
			\end{array}
			\end{equation*}
			using Lemmas \ref{p4-1stlemps}-\ref{p4-lem2} with $\ell=0$ and the fact that $v^1_x$, $v^3_x$ are uniformly bounded in $L^2(0,L)$ and $\|v^1\|=O(|\la|^{-1})$, $\|v^3\|=O(|\la|^{-1})$, we get 
		\begin{equation}\label{p4-4.23}
			l\rho_1 \int_{\varepsilon}^{\beta-\varepsilon}\f_2 |\la v^1|^2 dx +2lk_1  \int_{\varepsilon}^{\beta-\varepsilon}\f_2|v^1_x|^2 dx+\mathtt{I}_3 =o(1).
		\end{equation}
	Now,  using integration by parts and the fact that $\f_2 (\varepsilon)=\f_2 (\beta-\varepsilon)=0$, then using the fact that $v^1_x$ is uniformly bounded in $L^2 (0,L)$, $\|v^1\|=O(|\la|^{-1})$, we get
		\begin{equation}\label{p4-4.23*}
		\mathtt{I}_3=-lk_1 \int_{\varepsilon}^{\beta-\varepsilon}\f_2 |v^1_x|^2dx-lk_1 \int_{\varepsilon}^{\beta-\varepsilon}\f_2^\prime v^1_x\overline{v^1}=-lk_1 \int_{\varepsilon}^{\beta-\varepsilon}\f_2 |v^1_x|^2dx+o(1).
		\end{equation}
			Inserting \eqref{p4-4.23*} in \eqref{p4-4.23}, we obtain
			$$
				l\rho_1 \int_{\varepsilon}^{\beta-\varepsilon}\f_2 |\la v^1|^2 dx +lk_1  \int_{\varepsilon}^{\beta-\varepsilon}\f_2|v^1_x|^2 dx =o(1).
			$$ 
			Finally, from the above estimation and the definition of $\f_2$, we obtain \eqref{p4-4.19}. The proof is thus complete.
\end{proof}

\begin{lem}\label{p3-4thlemma}
 	{\rm If  \eqref{H3} holds. then the solution $U=(v^1,v^2,v^3,v^4,v^5,v^6)^{\top}\in D(\AA)$ of  \eqref{p3-f1ps}-\eqref{p3-f6ps} satisfies the following estimation
	\begin{equation}\label{p4-4.37}
	\int_{4\varepsilon}^{\beta-4\varepsilon}|\la v^1|^2 dx =o(\la^{-2}) \quad \text{and}\quad \int_{4\varepsilon}^{\beta-4\varepsilon}|v^1_x|^2 dx =o(\la^{-2}).
	\end{equation}
}
\end{lem}
\begin{proof} 
	For clarity, we divide the proof into three steps:\\\linebreak
	\textbf{Step 1}: In this step, we will prove that:
	\begin{equation}\label{p4-4.38}
	\begin{array}{lll}
	\displaystyle	l\rho_1 \int_{3\varepsilon}^{\beta-3\varepsilon} \f_4 |\la v^1|^2 dx+lk_3 \int_{3\varepsilon}^{\beta-3\varepsilon}\f_4 |v^1_x|^2dx+\Re\left\{-\rho_1 \la^2 \int_{3\varepsilon}^{\beta-3\varepsilon}\f_4 v^5\overline{v^1_x}dx+k_3\int_{3\varepsilon}^{\beta-3\varepsilon}\f_4 v^5_x \overline{v^1_{xx}}dx\right\}\vspace{0.25cm}\\
=\displaystyle\frac{o(1)}{\la^2}.
	\end{array}
\end{equation}	
For this aim, take $\ell=4$ in \eqref{p4-4.18*} and multiply it by $l\f_4 \overline{v^1} $, integrating over $(3\varepsilon,\beta-3\varepsilon)$, using the fact that $\|v^1\|=o(|\la|^{-1})$, $\|f^1\|=o(1)$ and $\|f^2\|=o(1)$, then taking the real part, we get 
	\begin{equation*}\label{p3-4.31}
	\begin{array}{lll}
	\displaystyle 	l\rho_1  \int_{3\varepsilon}^{\beta-3\varepsilon} \f_4 \left|\la v^1 \right|^2 dx +\underbrace{\Re \left\{lk_1 \int_{3\varepsilon}^{\beta-3\varepsilon} \f_4 (v^1_x +v^3 +lv^5)_x \overline{v^1 }dx  \right\}}_{:= \mathtt{I}_4} \displaystyle +\, \Re \left\{l^2k_3 \int_{3\varepsilon}^{\beta-3\varepsilon} \f_4 (v^5_x -lv^1 )\overline{v^1}dx  \right\}=\frac{o(1)}{\la^4}.
	\end{array}
	\end{equation*}
	From the above estimation, Lemma \ref{p4-lem2}-\ref{p4-lem3} with $\ell=4$, we obtain
	\begin{equation}\label{p4-4.39}
		l\rho_1 \int_{3\varepsilon}^{\beta-3\varepsilon} \f_4 |\la v^1|^2 dx +\mathtt{I}_4=\frac{o(1)}{\la^2}.
	\end{equation}
	Using integration by parts and the definition of $\f_4$, we obtain
	\begin{equation}\label{p3-4.32}
	\begin{array}{lll}
\displaystyle 	\mathtt{I}_4=-\Re \left\{lk_1 \int_{2\varepsilon}^{\beta-2\varepsilon} \g^{\prime} (v^1_x +v^3 +lv^5) \overline{v^1 }dx  \right\}\ -\Re \left\{lk_1 \int_{2\varepsilon}^{\beta-2\varepsilon} \g (v^1_x +v^3 +lv^5) \overline{v^1_x }dx  \right\}\vspace{0.25cm}\\
\hspace{0.5cm}\displaystyle =-\frac{lk_1}{2}\int_{2\varepsilon}^{\beta-2\varepsilon} \g^{\prime}\left(\left|v^1\right|^2\right)_x dx -\Re \left\{lk_1 \int_{2\varepsilon}^{\beta-2\varepsilon} \g^{\prime}v^3 \overline{v^1}dx  \right\}-\Re\left\{l^2 k_1 \int_{2\varepsilon}^{\beta-2\varepsilon} \g^{\prime}v^5 \overline{v^1}dx  \right\}\vspace{0.25cm}\\ 
\hspace{1cm}\displaystyle -\,\Re \left\{lk_1 \int_{2\varepsilon}^{\beta-2\varepsilon} \g (v^1_x +v^3 +lv^5) \overline{v^1_x }dx  \right\}.
%\hspace{0.5cm}\displaystyle -\,lk_1 \int_{2\varepsilon}^{\beta-2\varepsilon} \g \left|v^1_x\right|^2 dx -\Re\left\{lk_1 \int_{2\varepsilon}^{\beta-2\varepsilon} \g v^3 \overline{v^1_x}dx \right\}-\Re \left\{l^2 k_1 \int_{2\varepsilon}^{\beta-2\varepsilon} \g v^5 \overline{v^1_x}dx  \right\}.
	\end{array}
	\end{equation}
		Using integration by parts  and the fact that $\f_4^{\prime}(3\varepsilon)=\f_4^{\prime}(\beta-3\varepsilon)=0$, then using Lemma \ref{p4-lem3}, we obtain
	\begin{equation}\label{p3-4.33}
	-\frac{lk_1}{2}\int_{2\varepsilon}^{\beta-2\varepsilon} \g^{\prime}\left(\left|v^1\right|^2\right)_x dx =	\frac{lk_1}{2}\int_{2\varepsilon}^{\beta-2\varepsilon} \g^{\prime \prime}\left|v^1\right|^2 dx  =\frac{o(1)}{\la^2}.
	\end{equation}
		Using Lemma \ref{p4-lem3} with $\ell =4$ and the fact that $\|v^3\|=O(|\la|^{-1})$, $\|v^5\|=O(|\la|^{-1})$, we obtain
	\begin{equation}\label{p3-4.34}
	\begin{array}{lll}
	\displaystyle -\Re \left\{lk_1 \int_{2\varepsilon}^{\beta-2\varepsilon} \g^{\prime}v^3 \overline{v^1}dx  \right\}=o(\la^{-2}) \ \  \text{and} \ \ -\Re\left\{l^2 k_1 \int_{2\varepsilon}^{\beta-2\varepsilon} \g^{\prime}v^5 \overline{v^1}dx  \right\}=o(\la^{-2}).
	\end{array}
	\end{equation}
	Inserting \eqref{p3-4.33} and \eqref{p3-4.34} in \eqref{p3-4.32}, we obtain
\begin{equation}\label{p4-4.35}
	\mathtt{I}_4 =- \Re \left\{lk_1 \int_{2\varepsilon}^{\beta-2\varepsilon} \g (v^1_x +v^3 +lv^5) \overline{v^1_x }dx\right\}+o(\la^{-2}).
\end{equation}
	From \eqref{p4-4.18***}, we deduce that 
	$$
 -lk_1 (v^1_x+v^3+lv^5)=	-\la^2\rho_1 v^5 -k_3 (v^5_x-lv^1)_x+a(x)v^6-\rho_1 \la^{-\ell}f^6 -i\rho_1\la^{-\ell+1}  f^5,
	$$
Inserting the above equation in \eqref{p4-4.35}, then using the fact that $v^1_x$ is uniformly bounded in $L^2(0,L)$, $\|f^5\|=o(1)$. $\|f^6\|=o(1)$, we obtain 
\begin{equation}\label{p4-4.44}
	\mathtt{I}_4= \Re\left\{-\rho_1 \la^2 \int_{3\varepsilon}^{\beta-3\varepsilon}\f_4 v^5\overline{v^1_x}dx\right.\underbrace{-k_3 \int_{3\varepsilon}^{\beta-3\varepsilon}\f_4 (v^5_x-lv^1)_x\overline{v^1_x}dx}_{:=\mathtt{I}_5}+\left.a_0 \int_{3\varepsilon}^{\beta-3\varepsilon}\f_4 v^6\overline{v^1_x}dx \right\} +o(|\la|^{-3}).
\end{equation}
	Using integration by parts and the definition of $\f_4$, we obtain 
\begin{equation}\label{p4-4.45}
\begin{array}{lll}
	\mathtt{I}_5&=&\displaystyle -k_3 \int_{3\varepsilon}^{\beta-3\varepsilon}\f_4 v^5_{xx}\overline{v^1_x}dx+lk_3 \int_{3\varepsilon}^{\beta-3\varepsilon}\f_4 |v^1_x|^2 dx\vspace{0.25cm}\\  &=& \displaystyle k_3\int_{3\varepsilon}^{\beta-3\varepsilon}\f_4 v^5_x \overline{v^1_{xx}}dx+\underbrace{k_3\int_{3\varepsilon}^{\beta-3\varepsilon}\f_4^\prime v^5_x \overline{v^1_{x}}dx}_{:=\mathtt{I_6}}+lk_3 \int_{3\varepsilon}^{\beta-3\varepsilon}\f_4 |v^1_x|^2 dx,
	\end{array}
\end{equation}
	using integration by parts and the fact that $\f_4^{\prime}(3\varepsilon)=\f_4^\prime (\beta-3\varepsilon)=0$, then using Lemma \ref{p4-1stlemps} with $\ell=4$ and the fact that $v^1_x$ is uniformly bounded in $L^2(0,L)$, $\|v^1_{xx}\|=O(|\la|)$, we obtain 
\begin{equation}
\mathtt{I}_6 =-k_3 \int_{3\varepsilon}^{\beta-3\varepsilon}\f_4^{\prime\prime}v^5\overline{v^1_x}dx -k_3 \int_{3\varepsilon}^{\beta-3\varepsilon}\f_4^{\prime}v^5\overline{v^1_{xx}}dx =\frac{o(1)}{\la^2},
\end{equation}
consequently, we obtain
\begin{equation}\label{p4-4.38.}
	\mathtt{I}_5 =\displaystyle k_3\int_{3\varepsilon}^{\beta-3\varepsilon}\f_4 v^5_x \overline{v^1_{xx}}dx+lk_3 \int_{3\varepsilon}^{\beta-3\varepsilon}\f_4 |v^1_x|^2 dx+\frac{o(1)}{\la^2}.
\end{equation}
	Using Lemma \ref{p4-1stlemps} with $\ell=4$ and the fact that $v^1_x$ is uniformly bounded in $L^2(0,L)$, we obtain
	\begin{equation}\label{p4-4.46}
	a_0\int_{3\varepsilon}^{\beta-3\varepsilon}\f_4 v^6 \overline{v^1_x}dx=\frac{o(1)}{\la^2}.
	\end{equation}
Inserting \eqref{p4-4.38.} and \eqref{p4-4.46} in \eqref{p4-4.44}, we obtain
\begin{equation*}
	\mathtt{I}_4= \Re\left\{-\rho_1 \la^2 \int_{3\varepsilon}^{\beta-3\varepsilon}\f_4 v^5\overline{v^1_x}dx+k_3\int_{3\varepsilon}^{\beta-3\varepsilon}\f_4 v^5_x \overline{v^1_{xx}}dx\right\}+lk_3 \int_{3\varepsilon}^{\beta-3\varepsilon}\f_4 |v^1_x|^2dx+\frac{o(1)}{\la^2}.
\end{equation*}
Thus, by inserting the above equation in \eqref{p4-4.39}, we obtain \eqref{p4-4.38}.\\\linebreak
\textbf{Step 2}: In this step, we will prove that:
\begin{equation}\label{p4-4.51}
\Re\left\{k_3 \int_{3\varepsilon}^{\beta-3\varepsilon}\f_4 v^5_{x}\overline{v^1_{xx}}dx\right\}=\Re\left\{\frac{\la^2 \rho_1 k_3}{k_1} \int_{3\varepsilon}^{\beta-3\varepsilon}\f_4 v^5 \overline{v^1_x}dx\right\}+\frac{o(1)}{\la^2}.
\end{equation}
For this aim, take $\ell=4$ in \eqref{p4-4.18*} and multiply it by $\frac{k_3}{k_1}\f_4 \overline{v^5_x}$, integrating over $(3\varepsilon,\beta-3\varepsilon)$, then using the fact that $v^5_x$ is uniformly bounded in $L^2(0,L)$, $\|f^1\|=o(1)$, $\|f^2\|=o(1)$, we obtain
\begin{equation}
\begin{array}{lll}
\displaystyle	k_3 \int_{3\varepsilon}^{\beta-3\varepsilon}\f_4  v^1_{xx}\overline{v^5_x}dx =\underbrace{-\frac{\la^2 \rho_1 k_3}{k_1} \int_{3\varepsilon}^{\beta-3\varepsilon}\f_4 v^1 \overline{v^5_x}dx}_{:=\mathtt{I}_7}\ \ \underbrace{-k_3 \int_{3\varepsilon}^{\beta-3\varepsilon}\f_4 v^3_x\overline{v^5_x}dx}_{:=\mathtt{I}_8}\vspace{0.25cm}\\
	\displaystyle -\frac{k_3}{k_1}l(k_1+k_3) \int_{3\varepsilon}^{\beta-3\varepsilon}\f_4 |v^5_x|^2 dx +\frac{l^2 k_3^2}{k_1 }\int_{3\varepsilon}^{\beta-3\varepsilon}\f_4 v^1 \overline{v^5_x}dx +\frac{o(1)}{|\la|^{3}}.
	\end{array}
\end{equation}
From the above equation, Lemmas \ref{p4-lem2}-\ref{p4-lem3} with $\ell=4$, we obtain
\begin{equation}\label{p4-I67}
k_3 \int_{3\varepsilon}^{\beta-3\varepsilon}\f_4 v^1_{xx}\overline{v^5_x}dx=\mathtt{I}_7 +\mathtt{I}_8 +\frac{o(1)}{|\la|^{\frac{5}{2}}}.
\end{equation}
Using integration by parts and the definition of $\f_4$, then using Lemmas \ref{p4-1stlemps}, \ref{p4-lem3} with $\ell=4$, we obtain
\begin{equation}\label{p4-I6}
	\mathtt{I}_7 =\frac{\la^2 \rho_1 k_3}{k_1} \int_{3\varepsilon}^{\beta-3\varepsilon}\f_4 v^1_x \overline{v^5}dx+\frac{\la^2 \rho_1 k_3}{k_1} \int_{3\varepsilon}^{\beta-3\varepsilon}\f_4^\prime v^1 \overline{v^5}dx=\frac{\la^2 \rho_1 k_3}{k_1} \int_{3\varepsilon}^{\beta-3\varepsilon}\f_4 v^1_x \overline{v^5}dx+\frac{o(1)}{\la^2}.
\end{equation}
Using integration by parts and the definition of $\f_4$, then using Lemma \ref{p4-1stlemps} and the fact that $v^3_x$ is uniformly bounded in $L^2(0,L)$, $\|v^3_{xx}\|=O(|\la|)$, we get 
\begin{equation}\label{p4-I7}
\mathtt{I}_8= k_3 \int_{3\varepsilon}^{\beta-3\varepsilon}\f_4 v^3_{xx}\overline{v^5}dx+k_3 \int_{3\varepsilon}^{\beta-3\varepsilon}\f_4^\prime v^3_{x}\overline{v^5}dx=\frac{o(1)}{\la^2}.
\end{equation}
Inserting \eqref{p4-I7} and \eqref{p4-I6} in \eqref{p4-I67}, then taking the real part, we obtain \eqref{p4-4.51}.\\\linebreak
\textbf{Step 3:} In this step, we conclude the proof of \eqref{p4-4.37}. For this aim, inserting  \eqref{p4-4.51} in \eqref{p4-4.38}, then using Young's inequality and Lemma \ref{p4-1stlemps} with $\ell=4$, we deduce that
	\begin{equation}
\begin{array}{lll}
\displaystyle	l\rho_1 \int_{3\varepsilon}^{\beta-3\varepsilon} \f_4 |\la v^1|^2 dx+lk_3 \int_{3\varepsilon}^{\beta-3\varepsilon}\f_4 |v^1_x|^2dx=\displaystyle\Re\left\{\rho_1 \la^2\left(1-\frac{k_3}{k_1}\right) \int_{3\varepsilon}^{\beta-3\varepsilon}\f_4 v^5\overline{v^1_x}dx\right\}+\frac{o(1)}{\la^2}\vspace{0.25cm}\\
\displaystyle \leq \displaystyle\rho_1 \la^2 \left| 1-\frac{k_3}{k_1}\right|\int_{3\varepsilon}^{\beta-3\varepsilon}\f_4 |v^5||v^1_x|^2 dx +\frac{o(1)}{\la^2}\vspace{0.25cm}\\
\displaystyle = \displaystyle\int_{3\varepsilon}^{\beta-3\varepsilon} \left(\frac{\rho_1 \la^2}{\sqrt{lk_3}}\left|1-\frac{k_3}{k_1}\right|\sqrt{\f_4}|v^5|\right)\left(\sqrt{lk_3}\sqrt{\f_4}|v^1_x|\right)dx+\frac{o(1)}{\la^2}\vspace{0.25cm}\\
\leq  \displaystyle \underbrace{\frac{\rho_1^2}{2lk_3}\left(1-\frac{k_3}{k_1}\right)^2 \la^4 \int_{3\varepsilon}^{\beta-3\varepsilon}\f_4 |v^5|^2 dx}_{=o(\la^{-2})} +\frac{lk_3}{2}\int_{3\varepsilon}^{\beta-3\varepsilon}\f_4 |v^1_x|^2 dx +\frac{o(1)}{\la^2}.
\end{array}
\end{equation}	
Thus, from the above estimation, we deduce that
$$
	l\rho_1 \int_{3\varepsilon}^{\beta-3\varepsilon} \f_4 |\la v^1|^2 dx+\frac{lk_3}{2} \int_{3\varepsilon}^{\beta-3\varepsilon}\f_4 |v^1_x|^2dx=\frac{o(1)}{\la^2}.
$$
Finally, from the above estimation and the definition of $\f_4$, we obtain \eqref{p4-4.38}. The proof is thus complete. 
\end{proof}

\begin{lem}\label{p3-3rdlem}
	{\rm The solution $U=(v^1,v^2,v^3,v^4,v^5,v^6)^{\top}\in D(\AA)$ of  \eqref{p3-f1ps}-\eqref{p3-f6ps} satisfies the following estimations 
		\begin{equation}\label{p3-4.17**}
		\int_{3\varepsilon}^{\beta-3\varepsilon}\left|v^3_x\right|^2dx =o(1) \quad \text{and} \quad \int_{4\varepsilon}^{\beta-4\varepsilon}\left|\la v^3\right|^2dx =o(1) \ \ \text{if} \ \ \eqref{H1} \ \ \text{holds},
		\end{equation}
		\begin{equation}\label{p3-4.17*}
		\int_{4\varepsilon}^{\beta-4\varepsilon}\left|v^3_x\right|^2dx =o(1) \quad \text{and} \quad \int_{5\varepsilon}^{\beta-5\varepsilon}\left|\la v^3\right|^2dx =o(1) \ \ \ \text{if} \ \ \eqref{H2} \ \ \text{holds},
		\end{equation}
			\begin{equation}\label{p3-4.17}
		\int_{5\varepsilon}^{\beta-5\varepsilon}\left|v^3_x\right|^2dx =o(1) \quad \text{and} \quad \int_{6\varepsilon}^{\beta-6\varepsilon}\left|\la v^3\right|^2dx =o(1) \ \ \ \text{if} \ \ \eqref{H3} \ \ \text{holds}.
		\end{equation}
	}
\end{lem}
\begin{proof}
	For clarity, we divide the proof into four steps:\\\linebreak
	\textbf{Step 1:} In this step, we assume that \eqref{H1} or \eqref{H2} or \eqref{H3} holds and we will prove that:
	\begin{equation}\label{p3-4.23*}
	\begin{array}{lll}
	\displaystyle\frac{k_1}{\rho_1}\int_{\omega_j} \f_j \left|v^3_x\right|^2 dx =\left(\frac{k_2}{\rho_2}-\frac{k_1}{\rho_1}\right)\int_{\omega_j} \f_j v^1_{xx}\overline{v^3_x}dx
	+\la^2 \int_{\omega_j}\f_j^{\prime} v^1 \overline{v^3}dx+\frac{k_2}{\rho_2}\int_{\omega_j}\f_j^{\prime} v^1_{x}\overline{v^3_{x}}dx\vspace{0.25cm}\\
	\displaystyle+\, \frac{k_1}{\rho_2}\int_{\omega_j}\f_j  v^1_x (\overline{v^1_x} +\overline{v^3} +l\overline{v^5})dx
	-\frac{l}{\rho_1}(k_1+k_3)\int_{\omega_j}\f_j  v^5_x\overline{v^3_x}dx+o(1)
	\end{array}
	\end{equation}
	and 
	\begin{equation}\label{p4-4.23**}
	\rho_2\int_{\omega_j} \f_j \left|\la v^3\right|^2 dx=k_2 \int_{\omega_j}\f_j|v^3_x|^2 dx + o(1),
	\end{equation}
	where $\omega_j:=((j-1)\varepsilon,\beta+(1-j)\varepsilon)$ and $j\in\{1,\cdots,6\}$.
	For this aim, multiplying \eqref{p4-4.18*}  by $\rho_1^{-1}\f_j\overline{v^3_x}$, integrating over $\omega_j$, using integration by parts and the definition of $\f_j$, we obtain
	\begin{equation}\label{p4-4.21*}
	\begin{array}{lll}
	\displaystyle \frac{k_1}{\rho_1}\int_{\omega_j}\f_j \left|v^3_x\right|^2 dx =-\la^2 \int_{\omega_j}\f_j  v^1 \overline{v^3_x}dx-\frac{k_1}{\rho_1}\int_{\omega_j}\f_j  v^1_{xx}\overline{v^3_x}dx -\frac{l}{\rho_1}(k_1+k_3)\int_{\omega_j}\f_j  v^5_x\overline{v^3_x}dx   \vspace{0.25cm}\\ \displaystyle  +\,\frac{l^2 k_3}{\rho_1}\int_{\omega_j}\f_j  v^1\overline{v^3_x}dx -\rho_1 \la^{-\ell}\int_{\omega_j}\f_j f^2\overline{v^3_x}dx+i\rho_1 \la^{-\ell+1}\int_{\omega_j}\f_j f^1_x\overline{v^3}dx+i\rho_1 \la^{-\ell+1}\int_{\omega_j}\f_j^\prime f^1\overline{v^3}dx.
	\end{array}
	\end{equation}
	Using the fact that $v^3_x$ is uniformly bounded in $L^2(0,L)$, $\|v^1\|=O(|\la|^{-1})$, $\|v^3\|=O(|\la|^{-1})$, $\|f^1\|=o(1)$, $\|f^1_x\|=o(1)$ and $\|f^2\|=o(1)$, we get 
	\begin{equation}\label{p3-4.22}
	\begin{array}{lll}
	\displaystyle\frac{l^2 k_3}{\rho_1}\int_{\omega_j}\f_j  v^1\overline{v^3_x}dx=o(1), \ \ -\rho_1 \la^{-\ell}\int_{\omega_j}\f_j f^2\overline{v^3_x}dx=\frac{o(1)}{\la^{\ell}}, \ \ i\rho_1 \la^{-\ell+1}\int_{\omega_j}\f_j f^1_x\overline{v^3}dx=\frac{o(1)}{\la^{\ell}} \  \ \text{and}\vspace{0.25cm}\\
	\displaystyle
	i\rho_1 \la^{-\ell+1}\int_{\omega_j}\f_j^\prime f^1\overline{v^3}dx=\frac{o(1)}{\la^{\ell}}.
	\end{array}
	\end{equation}
	Inserting \eqref{p3-4.22} in \eqref{p4-4.21*} and using the fact that $\ell\geq 0$, we get 
	\begin{equation}\label{p3-4.23}
	\frac{k_1}{\rho_1}\int_{\omega_j} \f_j \left|v^3_x\right|^2 dx =-\la^2 \int_{\omega_j} \f_j v^1 \overline{v^3_x}dx-\frac{k_1}{\rho_1}\int_{\omega_j} \f_j v^1_{xx}\overline{v^3_x}dx-\frac{l}{\rho_1}(k_1+k_3)\int_{\omega_j}\f_j  v^5_x\overline{v^3_x}dx+o(1).
	\end{equation}
	Now, from \eqref{p4-4.18**}, we deduce that  
	\begin{equation}\label{p3-4.20}
	\la^2 \rho_2 \overline{v^3} +k_2 \overline{v^3_{xx}}-k_1 (\overline{v^1_x} +\overline{v^3}+l\overline{v^5})=-\rho_2 \la^{-\ell}\overline{f^4}+i\rho_2\la^{\ell-1}\overline{f^3}.
	\end{equation}
	Multiplying \eqref{p3-4.20} by $\rho_2^{-1}\f_jv^1_x$, integrating over $\omega_j$, using integration by parts and the definition of $\f_j$, then using the fact that $v^1_x$ is uniformly bounded in $L^2(0,L)$, $\|v^1\|=O(|\la|^{-1})$, $\|f^3\|=o(1)$, $\|f^3_x\|=o(1)$, $\|f^3\|=o(1)$, we obtain
	\begin{equation}
	\begin{array}{lll}
	\displaystyle	\la^2 \int_{\omega_j}\f_j v^1_x \overline{v^3 }dx +\frac{k_2}{\rho_2}\int_{\omega_j}\f_j v^1_{x}\overline{v^3_{xx}}dx -\frac{k_1}{\rho_2}\int_{\omega_j}\f_j v^1_x (\overline{v^1_x} +\overline{v^3} +l\overline{v^5})dx=\underbrace{-\rho_2 \la^{-\ell}\int_{\omega_j}\f_j\overline{f^4}v^1_xdx}_{=o(\la^{-\ell})}\vspace{0.25cm}\\
	\displaystyle	\underbrace{-i\rho_2\la^{1-\ell}\int_{\omega_j}\f_j\overline{f^3_x}v^1dx}_{=o(\la^{-\ell})}\ \ \underbrace{-\,i\rho_2\la^{1-\ell}\int_{\omega_j}\f_j^\prime\overline{f^3}v^1dx}_{=o(\la^{-\ell})}.
	\end{array}
	\end{equation}
	Using integration by parts to the first two terms in the above equation, we get 
	\begin{equation}\label{p3-4.25}
	\begin{array}{lll}
	\displaystyle-\la^2\int_{\omega_j}\f_j v^1 \overline{v^3_x }dx =\frac{k_2}{\rho_2}\int_{\omega_j}\f_j  v^1_{xx}\overline{v^3_{x}}dx+ \la^2 \int_{\omega_j}\f_j^{\prime} v^1 \overline{v^3}dx+\frac{k_2}{\rho_2}\int_{\omega_j}\f_j^{\prime} v^1_{x}\overline{v^3_{x}}dx\vspace{0.25cm}\\\displaystyle +\, \frac{k_1}{\rho_2}\int_{\omega_j}\f_j  v^1_x (\overline{v^1_x} +\overline{v^3} +l\overline{v^5})dx+o(\la^{-\ell}).
	\end{array}
	\end{equation}
	Inserting \eqref{p3-4.25} in \eqref{p3-4.23}, we obtain \eqref{p3-4.23*}.
	Next, multiplying \eqref{p3-4.20} by $\f_j v^3$, integrating over $\omega_j$, using integration by parts and the definition of $\f_j$ and the fact that $\|v^3\|=O(|\la|^{-1})$, $\|f^3\|=o(1)$ and $\|f^4\|=o(1)$, we get
	\begin{equation*}
	\rho_2 \int_{\omega_j} \f_j \left|\la v^3\right|^2 dx =k_2 \int_{\omega_j} \f_j |v^3_x|^2dx+k_2 \int_{\omega_j} \f_j^{\prime} \overline{v^3_{x}}v^3 dx +k_1 \int_{\omega_j} \f_j (\overline{v^1_x} +\overline{v^3}+l\overline{v^5})v^3 dx +o(\la^{-\ell}).
	\end{equation*}
	From the above estimation, the first estimation in \eqref{p3-4.17*} and the fact that $(v^1_x+v^3+lv^5)$, $v^3_x$ are uniformly bounded in $L^2 (0,L)$, $\|v^3\|=O(|\la|^{-1})$ and $\ell \geq 0$, we obtain  \eqref{p4-4.23**}.\\\linebreak
	\textbf{Step 2:} In this step, we assume that \eqref{H1} holds and we conclude the proof of  \eqref{p3-4.17**}. For this aim, take $j=3$ in \eqref{p3-4.23*} and using the fact that $\frac{k_1}{\rho_1}=\frac{k_2}{\rho_2}$, we get 
	\begin{equation}
	\begin{array}{lll}
	\displaystyle\frac{k_1}{\rho_1}\int_{2\varepsilon}^{\beta-2\varepsilon} \f_3 \left|v^3_x\right|^2 dx =
	\la^2 \int_{2\varepsilon}^{\beta-2\varepsilon} \f_3^{\prime} v^1 \overline{v^3}dx+\frac{k_2}{\rho_2}\int_{2\varepsilon}^{\beta-2\varepsilon} \f_3^{\prime} v^1_{x}\overline{v^3_{x}}dx\vspace{0.25cm}\\
	\displaystyle+\, \frac{k_1}{\rho_2}\int_{2\varepsilon}^{\beta-2\varepsilon} \f_3 v^1_x (\overline{v^1_x} +\overline{v^3} +l\overline{v^5})dx
	-\frac{l}{\rho_1}(k_1+k_3)\int_{2\varepsilon}^{\beta-2\varepsilon} \f_3  v^5_x\overline{v^3_x}dx+o(1).
	\end{array}
	\end{equation}
	Using Lemma \ref{p4-lem2} with $\ell=0$, Lemma \ref{p4-lem4}, the fact that $v^3_x$, $(v^1_x+v^3+lv^5)$ are uniformly bounded in $L^2(0,L)$ and $\|v^3\|=O(|\la|^{-1})$, and the definition of $\f_3$, we get the first estimation in \eqref{p3-4.17**}. Next, take $j=4$ in \eqref{p4-4.23**}, using the first estimation in \eqref{p3-4.17**} and the definition of $\f_4$, we obtain the second estimation in \eqref{p3-4.17**}.\\\linebreak
	\textbf{Step 3:} In this step, we assume that $\eqref{H2}$ holds and we conclude the proof of \eqref{p3-4.17*}. For this aim, take $j=4$ in \eqref{p3-4.23*} and using the fact that $\frac{k_1}{\rho_1}=\frac{k_2}{\rho_2}$, we get 
	\begin{equation}
	\begin{array}{lll}
	\displaystyle\frac{k_1}{\rho_1}\int_{3\varepsilon}^{\beta-3\varepsilon} \f_4 \left|v^3_x\right|^2 dx =
	\la^2 \int_{3\varepsilon}^{\beta-3\varepsilon} \f_4^{\prime} v^1 \overline{v^3}dx+\frac{k_2}{\rho_2}\int_{3\varepsilon}^{\beta-3\varepsilon} \f_4^{\prime} v^1_{x}\overline{v^3_{x}}dx\vspace{0.25cm}\\
	\displaystyle+\, \frac{k_1}{\rho_2}\int_{3\varepsilon}^{\beta-3\varepsilon} \f_4 v^1_x (\overline{v^1_x} +\overline{v^3} +l\overline{v^5})dx
	-\frac{l}{\rho_1}(k_1+k_3)\int_{3\varepsilon}^{\beta-3\varepsilon} \f_4 v^5_x\overline{v^3_x}dx+o(1).
	\end{array}
	\end{equation}
	Using  Lemma \ref{p4-lem2} with $\ell=2$, Lemma \ref{p4-lem3}, the fact that $v^3_x$, $(v^1_x+v^3+lv^5)$ are uniformly bounded in $L^2(0,L)$ and $\|v^3\|=O(|\la|^{-1})$, and the definition of $\f_4$, we get the first estimation in \eqref{p3-4.17*}. Next, take $j=5$ in \eqref{p4-4.23**}, using the first estimation in \eqref{p3-4.17*} and the definition of $\f_5$, we obtain the second estimation in \eqref{p3-4.17*}.\\\linebreak
	\textbf{Step 4:} In this step, we assume that \eqref{H3} holds and we conclude the proof of \eqref{p3-4.17}. For this aim, take $j=5$ in \eqref{p3-4.23*}, we get 
	\begin{equation}
	\begin{array}{lll}
	\displaystyle\frac{k_1}{\rho_1}\int_{4\varepsilon}^{\beta-4\varepsilon} \f_5 \left|v^3_x\right|^2 dx =\left(\frac{k_2}{\rho_2}-\frac{k_1}{\rho_1}\right)\int_{4\varepsilon}^{\beta-4\varepsilon} \f_5 v^1_{xx}\overline{v^3_x}dx
	+\la^2 \int_{4\varepsilon}^{\beta-4\varepsilon}\f_5^{\prime} v^1 \overline{v^3}dx+
	\la^2 \int_{4\varepsilon}^{\beta-4\varepsilon} \f_5^{\prime} v^1 \overline{v^3}dx\vspace{0.25cm}\\
	\displaystyle   +\,\frac{k_2}{\rho_2}\int_{4\varepsilon}^{\beta-4\varepsilon} \f_5^{\prime} v^1_{x}\overline{v^3_{x}}dx+ \frac{k_1}{\rho_2}\int_{4\varepsilon}^{\beta-4\varepsilon} \f_5 v^1_x (\overline{v^1_x} +\overline{v^3} +l\overline{v^5})dx
	-\frac{l}{\rho_1}(k_1+k_3)\int_{4\varepsilon}^{\beta-4\varepsilon} \f_5 v^5_x\overline{v^3_x}dx+o(1).
	\end{array}
	\end{equation}
		Using  Lemma \ref{p4-lem2} with $\ell=4$, Lemma \ref{p3-4thlemma}, the fact that $v^3_x$, $(v^1_x+v^3+lv^5)$ are uniformly bounded in $L^2(0,L)$ and $\|v^3\|=O(|\la|^{-1})$, and the definition of $\f_4$, we get
		$$
		\frac{k_1}{\rho_1}\int_{4\varepsilon}^{\beta-4\varepsilon} \f_5 \left|v^3_x\right|^2 dx =\left(\frac{k_2}{\rho_2}-\frac{k_1}{\rho_1}\right)\int_{4\varepsilon}^{\beta-4\varepsilon} \f_5 v^1_{xx}\overline{v^3_x}dx+o(1).
		$$
	Using integration by parts in the above equation and the fact that $\f_5 (4\varepsilon)=\f_5 (\beta-4\varepsilon)=0$, we get
		$$
	\frac{k_1}{\rho_1}\int_{4\varepsilon}^{\beta-4\varepsilon} \f_5 \left|v^3_x\right|^2 dx =\left(\frac{k_2}{\rho_2}-\frac{k_1}{\rho_1}\right)\int_{4\varepsilon}^{\beta-4\varepsilon} \f_5 v^1_{x}\overline{v^3_{xx}}dx+\left(\frac{k_2}{\rho_2}-\frac{k_1}{\rho_1}\right)\int_{4\varepsilon}^{\beta-4\varepsilon} \f_5^\prime v^1_{x}\overline{v^3_{x}}dx+o(1).
	$$
	From the above estimation, Lemma \ref{p3-4thlemma} and the fact that $v^3_x$ is uniformly bounded in $L^2(0,L)$, $\|v^3_{xx}\|=O(|\la|)$, and the definition of $\f_5$, we get the first estimation in \eqref{p3-4.17}. Finally, take $j=6$ in \eqref{p4-4.23**}, using the first estimation in \eqref{p3-4.17} and the definition of $\f_6$, we obtain the second estimation in \eqref{p3-4.17}. The proof is thus complete.
\end{proof}

	\begin{lem}\label{p4-lem7}
	{\rm
		The solution $U=(v^1,v^2,v^3,v^4,v^5,v^6)^{\top}\in D(\AA)$ of system \eqref{p3-f1ps}-\eqref{p3-f6ps} satisfies the following estimations
		\begin{equation}\label{p4-4.63}
		\mathsf{J}(4\varepsilon,\beta-4\varepsilon)=o(1) \qquad \text{if} \qquad  \eqref{H1} \ \ \text{holds},
		\end{equation}
		\begin{equation}\label{p4-4.64}
		\mathsf{J}(5\varepsilon,\beta-5\varepsilon)=o(1) \qquad \text{if} \qquad  \eqref{H2} \ \ \text{holds},
		\end{equation}
			\begin{equation}\label{p4-4.65}
		\mathsf{J}(6\varepsilon,\beta-6\varepsilon)=o(1) \qquad \text{if} \qquad  \eqref{H3} \ \ \text{holds},
		\end{equation}
		where 
		\begin{equation*}
		\begin{array}{lll}
		\displaystyle \mathsf{J}(\gamma_1,\gamma_2):=	\int_{0}^{\alpha_1} \left(\rho_1\left|\la v^1\right|^2+k_1\left|v^1_x \right|^2+\rho_2\left| \la v^3 \right|^2 +k_2 \left|v^3_x\right|^2 +k_3|v^5_x|^2 \right) dx   \vspace{0.25cm}\\
		\hspace{2cm}	\displaystyle
		+	\int_{\alpha_2}^{L} \left(\rho_1\left|\la v^1\right|^2+k_1\left|v^1_x \right|^2+\rho_2\left| \la v^3 \right|^2 +k_2 \left|v^3_x\right|^2+k_3 |v^5_x|^2 \right) dx +\rho_1\int_{\beta}^{L} \left|\la v^5\right|^2 dx,
		\end{array}
		\end{equation*}
		for all $0<\alpha_1<\alpha_2<\beta<L$.
	}
\end{lem}
\begin{proof} We divide the proof into two steps :\\ \linebreak
	\textbf{Step 1:} Let $\h \in C^1([0,L])$ such that $\h (0)=\h(L)=0$. In this step, we assume that \eqref{H1} or \eqref{H2} or \eqref{H3} holds and we will prove that:
	\begin{equation}\label{p3-4.41}
		\intdx \h^{\prime} \left(\rho_1\left|\la v^1\right|^2+k_1\left|v^1_x \right|^2+\rho_2\left| \la v^3 \right|^2 +k_2 \left|v^3_x\right|^2+\rho_1 \left|\la v^5\right|^2+k_3\left| v^5_x \right|^2 \right) dx=o(1).
	\end{equation}
	For this aim, multiplying \eqref{p4-4.18*} by $2\h \overline{v^1_x}$, integrating over $(0,L)$, taking the real part, using integration by parts and the definition of $\h$, then using the fact that $v^1_x$ is uniformly bounded in $L^2(0,L)$, $\|v^1\|=O(|\la|^{-1})$, $\|f^1\|=o(1)$, $\|f^1_x\|=o(1)$ and $\|f^2\|=o(1)$, we obtain 
	\begin{equation}\label{p3-4.42*}
	\begin{array}{lll}
\displaystyle 	\intdx \h \left(\rho_1\left|\la v^1\right|^2+k_1\left|v^1_x \right|^2 \right)_x dx +\Re \left\{2k_1 \intdx \h v^3_x \overline{v^1_x}dx  \right\}+\Re\left\{2l(k_1+k_3)\intdx \h v^5_x \overline{v^1_x}dx \right\}\vspace{0.25cm}\\ \displaystyle
\underbrace{	-\Re\left\{2l^2k_3 \intdx \h v^1\overline{v^1_x}dx  \right\}}_{=o(1)}=\underbrace{\Re\left\{-\frac{\rho_1}{\la^{\ell}}\intdx \h f^2\overline{v^1_x}dx\right\}}_{=o(\la^{-\ell})}+\underbrace{\Re \left\{\frac{i\rho_1}{\la^{\ell-1}} \intdx \h f^1_x \overline{v^1}dx\right\}}_{=o(\la^{-\ell})}\vspace{0.25cm}\\
\displaystyle +\underbrace{\Re \left\{\frac{i\rho_1}{\la^{\ell-1}}\intdx \h^\prime f^1 \overline{v^1}dx\right\}}_{=o(\la^{-\ell})}.
	\end{array}
	\end{equation}
	Now, multiplying \eqref{p4-4.18**} by $2\h \overline{v^3_x}$, integrating over $(0,L)$, taking the real part, using integration by parts and the definition of $\h$, then using the fact that $v^3_x$ is uniformly bounded in $L^2(0,L)$, $\|v^3\|=O(|\la|^{-1})$, $\|v^5\|=O(|\la|^{-1})$, $\|f^3\|=o(1)$, $\|f^3_x\|=o(1)$ and $\|f^4\|=o(1)$, we obtain 
	 \begin{equation}\label{p3-4.43*}
	 \begin{array}{lll}
	 \displaystyle	\intdx \h \left(\left|\rho_2 \la v^3 \right|^2 +k_2 \left|v^3_x\right|^2\right)_x dx -\Re\left\{2k_1 \intdx \h v^1_x \overline{v^3_x}dx  \right\}\underbrace{-\Re \left\{2k_1\intdx\h (v^3+lv^5)\overline{v^3_x }dx  \right\}}_{=o(1)}\vspace{0.25cm}\\
	 \displaystyle =\underbrace{\Re\left\{-\frac{\rho_2}{\la^{\ell}}\intdx \h f^4\overline{v^3_x}dx\right\}}_{=o(\la^{-\ell})}+\underbrace{\Re \left\{\frac{i\rho_2}{\la^{\ell-1}} \intdx \h f^3_x \overline{v^3}dx\right\}}_{=o(\la^{-\ell})}+\underbrace{\Re \left\{\frac{i\rho_2}{\la^{\ell-1}}\intdx \h^\prime f^3 \overline{v^3}dx\right\}}_{=o(\la^{-\ell})}.
	 	\end{array}
	 \end{equation}
		Next multiplying \eqref{p4-4.18***} by $2\h \overline{v^5_x}$, integrating over $(0,L)$, taking the real part, using integration by parts and the definition of $\h$,  then using the definition of $a(x)$, Lemma \ref{p4-1stlemps}, the fact that $v^5_x$ is uniformly bounded in $L^2(0,L)$, $\|v^3\|=O(|\la|^{-1})$, $\|v^5\|=O(|\la|^{-1})$,  $\|f^5\|=o(1)$, $\|f^5_x\|=o(1)$ and $\|f^6\|=o(1)$, we obtain
	\begin{equation}\label{p3-4.47*}
\begin{array}{lll}
\displaystyle 	\intdx \h \left(\left|\rho_1 \la v^5 \right|^2 +k_3 \left|v^5_x\right|^2\right)_x dx -\Re \left\{2l(k_1+k_3)\intdx \h v^1_x \overline{v^5_x}dx \right\}\underbrace{-\Re \left\{2lk_1 \intdx \h (v^3+lv^5)\overline{v^5_x}dx  \right\}}_{=o(1)}\vspace{0.25cm}\\ 
\displaystyle \underbrace{-\Re \left\{a_0 \int_0^\beta \h v^6\overline{v^5_x}dx \right\}}_{=o\left(|\la|^{-\frac{\ell}{2}}\right)} =\underbrace{\Re\left\{-\frac{\rho_1}{\la^{\ell}}\intdx \h f^6\overline{v^5_x}dx\right\}}_{=o(\la^{-\ell})}+\underbrace{\Re \left\{\frac{i\rho_1}{\la^{\ell-1}} \intdx \h f^5_x \overline{v^5}dx\right\}}_{=o(\la^{-\ell})}\vspace{0.25cm}\\
\displaystyle +\underbrace{\Re \left\{\frac{i\rho_1}{\la^{\ell-1}}\intdx \h^\prime f^5 \overline{v^5}dx\right\}}_{=o(\la^{-\ell})}.
\end{array}
	\end{equation}
	Adding \eqref{p3-4.42*}, \eqref{p3-4.43*}, \eqref{p3-4.47*} and using the fact that $\ell \in \{0,2,4\}$, then using integration by parts, we obtain  \eqref{p3-4.41}.\\\linebreak 
	\textbf{Step 2:} In this step, we conclude the proof of Lemma \ref{p4-lem7}. For this aim, take $\h =x\q_1+(x-L)\q_2$ in \eqref{p3-4.41}, we obtain
	\begin{equation*}
	\begin{array}{lll}
	\displaystyle	\int_{0}^{\alpha_1} \left(\rho_1\left|\la v^1\right|^2+k_1\left|v^1_x \right|^2+\rho_2\left| \la v^3 \right|^2 +k_2 \left|v^3_x\right|^2 +k_3 |v^5_x|^2 \right) dx   \vspace{0.25cm}\\
	\displaystyle
	+	\int_{\alpha_2}^{L} \left(\rho_1\left|\la v^1\right|^2+k_1\left|v^1_x \right|^2+\rho_2\left| \la v^3 \right|^2 +k_2 \left|v^3_x\right|^2+k_3 \left|v^5_x\right|^2 \right) dx +\rho_1\int_{\beta}^{L} \left|\la v^5\right|^2 dx\vspace{0.25cm}\\ 
	\displaystyle = -\int_{\alpha_1}^{\alpha_2}\left(\q_1+x\q_1^{\prime}\right)\left(\rho_1\left|\la v^1\right|^2+k_1\left|v^1_x \right|^2+\rho_2\left| \la v^3 \right|^2 +k_2 \left|v^3_x\right|^2+k_3 \left| v^5_x\right|^2 \right) dx\vspace{0.25cm}\\
	\hspace{0.5cm}\displaystyle  -\int_{\alpha_1}^{\alpha_2}\left(\q_2+(x-L)\q_2^{\prime}\right)\left(\rho_1\left|\la v^1\right|^2+k_1\left|v^1_x \right|^2+\rho_2\left| \la v^3 \right|^2 +k_2 \left|v^3_x\right|^2+k_3 \left| v^5_x\right|^2 \right) dx\vspace{0.25cm}\\
	\hspace{0.5cm}\displaystyle +\rho_1\int_{0}^{\alpha_2} \q_1 |\la v^5|^2dx
	+\rho_1 \int_{\alpha_1}^{\beta} \q_2 |\la v^5|^2dx.
	\end{array}
	\end{equation*}
	Now, take $\alpha_1=4\varepsilon$ and $\alpha_2=\beta-4\varepsilon$ in the above equation, then using Lemmas \ref{p4-1stlemps}, \ref{p4-lem2}, \ref{p4-lem4} in case that \eqref{H1} holds and \eqref{p3-4.17**}, we obtain \eqref{p4-4.63}.
Next, take $\alpha_1=5\varepsilon$ and $\alpha_2=\beta-5\varepsilon$ in the above equation, then using Lemmas \ref{p4-1stlemps}-\ref{p4-lem3} in  case that \eqref{H2} holds and \eqref{p3-4.17*}, we obtain \eqref{p4-4.64}. Finally, take $\alpha_1=6\varepsilon$ and $\alpha_2=\beta-6\varepsilon$ in the above equation, then using Lemmas \ref{p4-1stlemps}, \ref{p4-lem2}, \ref{p3-4thlemma} in case that \eqref{H3} holds and \eqref{p3-4.17}, we obtain \eqref{p4-4.65}. The proof is thus complete.
\end{proof}
\\\linebreak
\textbf{Proof of Theorem \ref{exps}}. 
First, from  Lemmas \ref{p4-1stlemps}, \ref{p4-lem2}, \ref{p4-lem4}, \eqref{p3-4.17**}, and the fact that $\ell=0$, we obtain 
\begin{equation}\label{p3-4.46***}
\left\{\begin{array}{l}
\displaystyle{ \int_0^{\beta}\abs{v^6}^2dx=o(1),\ \int_0^{\beta}\abs{v^5}^2dx=o(\la^{-2}),\ \int_{\varepsilon}^{\beta-\varepsilon}}|v^5_x|^2 dx =o(1), \ \int_{2\varepsilon}^{\beta-2\varepsilon}|v^1_x|^2 dx =o(1)\vspace{0.25cm}\\
\displaystyle{\int_{2\varepsilon}^{\beta-2\varepsilon}\abs{\la v^1}^2dx=o(1), \ \int_{3\varepsilon}^{\beta-3\varepsilon}\abs{v^3_x}^2dx=o(1) \quad \text{and}\quad \int_{4\varepsilon}^{\beta-4\varepsilon}\abs{\la v^3}^2dx=o(1)}.
\end{array}
\right.
\end{equation}
From \eqref{p3-4.46***}, \eqref{p4-4.63} and the fact that $\displaystyle 0<\varepsilon<\frac{\beta}{12}$, we deduce that $\|U\|_\HH=o(1) $, which contradicts \eqref{p3-H-cond}. This implies that 
$$
\sup_{\la\in \R}\|(i\la I-\AA)^{-1}\|_{\mathcal{\HH}}=O\left(1\right).
$$ 
The proof is thus complete.\xqed{$\square$}
\\\linebreak
\textbf{Proof of Theorem \ref{p3-pol-eq}}. 
First, from  Lemmas \ref{p4-1stlemps}, \ref{p4-lem2}, \ref{p4-lem3}, \eqref{p3-4.17*}, and the fact that $\ell=2$, we obtain 
\begin{equation}\label{p3-4.46*}
\left\{\begin{array}{l}
\displaystyle{ \int_0^{\beta}\abs{v^6}^2dx=o(\la^{-2}),\ \int_0^{\beta}\abs{v^5}^2dx=o(\la^{-4}),\ \int_{\varepsilon}^{\beta-\varepsilon}|v^5_x|^2 dx =o(\la^{-2})},\ \int_{2\varepsilon}^{\beta-2\varepsilon}|v^1_x|^2 dx =o(1)\vspace{0.25cm}\\
\displaystyle{\int_{3\varepsilon}^{\beta-3\varepsilon}\abs{\la v^1}^2dx=o(1), \ \int_{4\varepsilon}^{\beta-4\varepsilon}\abs{v^3_x}^2dx=o(1) \quad \text{and}\quad \int_{5\varepsilon}^{\beta-5\varepsilon}\abs{\la v^3}^2dx=o(1)}.
\end{array}
\right.
\end{equation}
From \eqref{p3-4.46*}, \eqref{p4-4.64} and the fact that $\displaystyle 0<\varepsilon<\frac{\beta}{12}$, we deduce that $\|U\|_\HH=o(1) $, which contradicts \eqref{p3-H-cond}. This implies that 
$$
\sup_{\la\in \R}\|(i\la I-\AA)^{-1}\|_{\mathcal{\HH}}=O\left(\la^2\right).
$$ 
The proof is thus complete.\xqed{$\square$}
\\\linebreak
\textbf{Proof of Theorem \ref{p3-pol-neq}}. 
First, from  Lemmas \ref{p4-1stlemps}, \ref{p4-lem2}, \ref{p3-4thlemma}, \eqref{p3-4.17}, and the fact that $\ell=4$, we obtain 
\begin{equation}\label{p3-4.46**}
\left\{\begin{array}{l}
\displaystyle{ \int_0^{\beta}\abs{v^6}^2dx=o(\la^{-4}),\ \int_0^{\beta}\abs{v^5}^2dx=o(\la^{-6}), \ \int_{\varepsilon}^{\beta-\varepsilon}|v^5_x|^2 dx =o(|\la|^{-3}), \  \int_{4\varepsilon}^{\beta-4\varepsilon}| v^1_x|^2 dx =o(\la^{-2})}\vspace{0.25cm}\\
\displaystyle{\int_{4\varepsilon}^{\beta-4\varepsilon}\abs{ \la v^1}^2dx=o(\la^{-2}), \ \int_{5\varepsilon}^{\beta-5\varepsilon}\abs{v^3_x}^2dx=o(1) \quad \text{and}\quad \int_{6\varepsilon}^{\beta-6\varepsilon}\abs{\la v^3}^2dx=o(1)}.
\end{array}
\right.
\end{equation}
From \eqref{p3-4.46**}, \eqref{p4-4.65} and the fact that $\displaystyle 0<\varepsilon<\frac{\beta}{12}$, we deduce that $\|U\|_\HH=o(1) $, which contradicts \eqref{p3-H-cond}. This implies that 
$$
\sup_{\la\in \R}\|(i\la I-\AA)^{-1}\|_{\mathcal{\HH}}=O\left(\la^4\right).
$$ 
The proof is thus complete.\xqed{$\square$}
	
%%%%%%%%%%%%%%%%%%%
%conclusion&&&&
%%%%%%%%%%%%%%%%%
\section{Conclusion}
We have studied  the stabilization of a Bresse system with discontinuous local viscoelastic damping of Kelvin-Voigt type acting in the longitudinal displacement under fully Dirichlet boundary conditions. We proved the strong stability of the system. We established the exponential stability of the solution if and only if the three waves have the same speed of propagation (i.e., $\frac{k_1}{\rho_1}=\frac{k_2}{\rho_2}$ and $k_1=k_3$). On the contrary,   we proved that the energy of our system  decays polynomially with the rates
\begin{equation*}
\left\{	\begin{array}{lll}
\displaystyle	t^{-1} \quad \text{if} \quad \frac{k_1}{\rho_1}=\frac{k_2}{\rho_2} \ \ \text{and} \ \ k_1\neq k_3,\vspace{0.15cm}\\
\displaystyle 		t^{-\frac{1}{2}} \quad \text{if} \quad \frac{k_1}{\rho_1}\neq\frac{k_2}{\rho_2}.
\end{array}
\right.
\end{equation*}
Moreover, it would be interesting to study system \eqref{p3-sysorig}-\eqref{p3-initialcond} with local internal frictional damping, in other words, by only assuming that $a$ is positive on a non empty subinterval of $(0,L)$ that could be away from the boundary.
%\section*{Data Availability}
%The data that support the finding of this study are available within the article.


\begin{thebibliography}{10}

\bibitem{Wehbe08}
F.~Abdallah, M.~Ghader, and A.~Wehbe.
\newblock \href{https://doi.org/10.1002/mma.4717} {Stability results of a
  distributed problem involving {B}resse system with history and/or Cattaneo
  law under fully Dirichlet or mixed boundary conditions}.
\newblock {\em Mathematical Methods in the Applied Sciences}, 41(5):1876--1907,
  jan 2018.

\bibitem{doi:10.1002/mma.6070}
M.~Afilal, A.~Guesmia, A.~Soufyane, and M.~Zahri.
\newblock On the exponential and polynomial stability for a linear {B}resse
  system.
\newblock {\em Mathematical Methods in the Applied Sciences}, 43(5):2626--2645,
  2020.

\bibitem{Akil2020}
M.~Akil, Y.~Chitour, M.~Ghader, and A.~Wehbe.
\newblock Stability and exact controllability of a {T}imoshenko system with
  only one fractional damping on the boundary.
\newblock {\em Asymptotic Analysis}, 119:221--280, 2020.
\newblock 3-4.

\bibitem{Alabau04}
F.~Alabau.
\newblock \href{https://doi.org/10.1016/s0764-4442(99)80316-4} {Stabilisation
  fronti{\`{e}}re indirecte de syst{\`{e}}mes faiblement coupl{\'{e}}s}.
\newblock {\em Comptes Rendus de l'Acad{\'{e}}mie des Sciences - Series I -
  Mathematics}, 328(11):1015--1020, June 1999.

\bibitem{Alabau02}
F.~Alabau, P.~Cannarsa, and V.~Komornik.
\newblock \href{https://doi.org/10.1007/s00028-002-8083-0} {Indirect internal
  stabilization of weakly coupled evolution equations}.
\newblock {\em Journal of Evolution Equations}, 2(2):127--150, May 2002.

\bibitem{Alabau05}
F.~Alabau-Boussouira.
\newblock \href{https://doi.org/10.1137/s0363012901385368} {Indirect Boundary
  Stabilization of Weakly Coupled Hyperbolic Systems}.
\newblock {\em {SIAM} Journal on Control and Optimization}, 41(2):511--541,
  Jan. 2002.

\bibitem{Alabau-Boussouira2007}
F.~Alabau-Boussouira.
\newblock Asymptotic behavior for {T}imoshenko beams subject to a single
  nonlinear feedback control.
\newblock {\em Nonlinear Differential Equations and Applications NoDEA},
  14(5):643--669, Dec 2007.

\bibitem{Alabau06}
F.~Alabau-Boussouira and M.~L{\'{e}}autaud.
\newblock \href{https://doi.org/10.1051/cocv/2011106} {Indirect stabilization
  of locally coupled wave-type systems}.
\newblock {\em {ESAIM}: Control, Optimisation and Calculus of Variations},
  18(2):548--582, Sept. 2011.

\bibitem{ALABAUBOUSSOUIRA2011481}
F.~{Alabau Boussouira}, J.~E. {Muñoz Rivera}, and D.~da~S.~{Almeida Júnior}.
\newblock Stability to weak dissipative {B}resse system.
\newblock {\em Journal of Mathematical Analysis and Applications}, 374(2):481
  -- 498, 2011.

\bibitem{https://doi.org/10.1002/mma.3115}
M.~Alves, L.~Fatori, M.~Jorge~Silva, and R.~Monteiro.
\newblock Stability and optimality of decay rate for a weakly dissipative
  bresse system.
\newblock {\em Mathematical Methods in the Applied Sciences}, 38(5):898--908,
  2015.

\bibitem{Arendt01}
W.~Arendt and C.~J.~K. Batty.
\newblock \href {http://dx.doi.org/10.2307/2000826}{Tauberian theorems and
  stability of one-parameter semigroups}.
\newblock {\em Trans. Amer. Math. Soc.}, 306(2):837--852, 1988.

\bibitem{BASSAM20151177}
M.~Bassam, D.~Mercier, S.~Nicaise, and A.~Wehbe.
\newblock Polynomial stability of the {T}imoshenko system by one boundary
  damping.
\newblock {\em Journal of Mathematical Analysis and Applications}, 425(2):1177
  -- 1203, 2015.

\bibitem{doi:10.1002/zamm.201500172}
M.~Bassam, D.~Mercier, S.~Nicaise, and A.~Wehbe.
\newblock Stability results of some distributed systems involving
  mindlin-{T}imoshenko plates in the plane.
\newblock {\em ZAMM - Journal of Applied Mathematics and Mechanics /
  Zeitschrift für Angewandte Mathematik und Mechanik}, 96(8):916--938, 2016.

\bibitem{Batty01}
C.~J.~K. Batty and T.~Duyckaerts.
\newblock \href {http://dx.doi.org/10.1007/s00028-008-0424-1} {Non-uniform
  stability for bounded semi-groups on {B}anach spaces}.
\newblock {\em J. Evol. Equ.}, 8(4):765--780, 2008.

\bibitem{Borichev01}
A.~Borichev and Y.~Tomilov.
\newblock \href{http://dx.doi.org/10.1007/s00208-009-0439-0} {Optimal
  polynomial decay of functions and operator semigroups}.
\newblock {\em Math. Ann.}, 347(2):455--478, 2010.

\bibitem{deLima2018}
P.~R. de~Lima and H.~D. Fern{\'a}ndez~Sare.
\newblock Stability of thermoelastic {B}resse systems.
\newblock {\em Zeitschrift f{\"u}r angewandte Mathematik und Physik}, 70(1):3,
  Nov 2018.

\bibitem{ElArwadi2019}
T.~El~Arwadi and W.~Youssef.
\newblock On the stabilization of the {B}resse beam with kelvin--voigt damping.
\newblock {\em Applied Mathematics {\&} Optimization}, Sep 2019.

\bibitem{doi:10.1080/00036811.2018.1520982}
L.~H. Fatori, M.~de~Oliveira~Alves, and H.~D.~F. Sare.
\newblock Stability conditions to {B}resse systems with indefinite memory
  dissipation.
\newblock {\em Applicable Analysis}, 99(6):1066--1084, 2020.

\bibitem{FATORI2012600}
L.~H. Fatori and R.~N. Monteiro.
\newblock The optimal decay rate for a weak dissipative {B}resse system.
\newblock {\em Applied Mathematics Letters}, 25(3):600 -- 604, 2012.

\bibitem{10.1093/imamat/hxq038}
L.~H. Fatori and J.~E. Muñoz~Rivera.
\newblock {Rates of decay to weak thermoelastic {B}resse system}.
\newblock {\em IMA Journal of Applied Mathematics}, 75(6):881--904, 06 2010.

\bibitem{GHOUL20171870}
T.-E. Ghoul, M.~Khenissi, and B.~Said-Houari.
\newblock On the stability of the {B}resse system with frictional damping.
\newblock {\em Journal of Mathematical Analysis and Applications},
  455(2):1870--1898, 2017.

\bibitem{Guesmia2017}
A.~Guesmia.
\newblock Asymptotic stability of {B}resse system with one infinite memory in
  the longitudinal displacements.
\newblock {\em Mediterranean Journal of Mathematics}, 14(2):49, Mar 2017.

\bibitem{doi:10.1002/mma.3228}
A.~Guesmia and M.~Kafini.
\newblock {B}resse system with infinite memories.
\newblock {\em Mathematical Methods in the Applied Sciences},
  38(11):2389--2402, 2015.

\bibitem{Huang01}
F.~L. Huang.
\newblock Characteristic conditions for exponential stability of linear
  dynamical systems in {H}ilbert spaces.
\newblock {\em Ann. Differential Equations}, 1(1):43--56, 1985.

\bibitem{Kato01}
T.~Kato.
\newblock {\em \href{https://doi.org/10.1007/978-3-642-66282-9} {Perturbation
  Theory for Linear Operators}}.
\newblock Springer Berlin Heidelberg, 1995.

\bibitem{LagneseLeugering01}
J.~E. Lagnese, G.~Leugering, and E.~J. P.~G. Schmidt.
\newblock {\em \href {https://doi.org/10.1007/978-1-4612-0273-8} {Modeling,
  Analysis and Control of Dynamic Elastic Multi-Link Structures}}.
\newblock Birkh\"{a}user Boston, 1994.

\bibitem{RaoLiu01}
Z.~Liu and B.~Rao.
\newblock \href{http://dx.doi.org/10.1007/s00033-004-3073-4} {Characterization
  of polynomial decay rate for the solution of linear evolution equation}.
\newblock {\em Z. Angew. Math. Phys.}, 56(4):630--644, 2005.

\bibitem{LiuRao01}
Z.~Liu and B.~Rao.
\newblock \href{http://dx.doi.org/10.1016/j.jmaa.2007.02.021}{Frequency domain
  approach for the polynomial stability of a system of partially damped wave
  equations}.
\newblock {\em J. Math. Anal. Appl.}, 335(2):860--881, 2007.

\bibitem{RaoLiu03}
Z.~Liu and B.~Rao.
\newblock \href {http://dx.doi.org/10.1007/s00033-008-6122-6}{Energy decay rate
  of the thermoelastic {B}resse system}.
\newblock {\em Z. Angew. Math. Phys.}, 60(1):54--69, 2009.

\bibitem{Rivera2019}
J.~E. Mu{\~{n}}oz~Rivera and M.~G. Naso.
\newblock Boundary stabilization of bresse systems.
\newblock {\em Zeitschrift f{\"u}r angewandte Mathematik und Physik}, 70(2):56,
  Mar 2019.

\bibitem{Wehbe03}
N.~Najdi and A.~Wehbe.
\newblock Weakly locally thermal stabilization of {B}resse systems.
\newblock {\em Electron. J. Differential Equations}, pages No. 182, 19, 2014.

\bibitem{Wehbe02}
N.~Noun and A.~Wehbe.
\newblock \href {http://dx.doi.org/10.1016/j.crma.2012.04.003}{Stabilisation
  faible interne locale de syst\`eme \'elastique de {B}resse}.
\newblock {\em C. R. Math. Acad. Sci. Paris}, 350(9-10):493--498, 2012.

\bibitem{Pazy01}
A.~Pazy.
\newblock {\em \href {http://dx.doi.org/10.1007/978-1-4612-5561-1} {Semigroups
  of linear operators and applications to partial differential equations}},
  volume~44 of {\em Applied Mathematical Sciences}.
\newblock Springer-Verlag, New York, 1983.

\bibitem{pruss01}
J.~Pr\"uss.
\newblock \href {http://dx.doi.org/10.2307/1999112} {On the spectrum of
  {$C_{0}$}-semigroups}.
\newblock {\em Trans. Amer. Math. Soc.}, 284(2):847--857, 1984.

\bibitem{Russell01}
D.~L. Russell.
\newblock \href{http://dx.doi.org/10.1006/jmaa.1993.1071} {A general framework
  for the study of indirect damping mechanisms in elastic systems}.
\newblock {\em J. Math. Anal. Appl.}, 173(2):339--358, 1993.

\bibitem{SORIANO2014369}
J.~Soriano, W.~Charles, and R.~Schulz.
\newblock Asymptotic stability for bresse systems.
\newblock {\em Journal of Mathematical Analysis and Applications},
  412(1):369--380, 2014.

\bibitem{doi:10.1080/00036810903156149}
A.~Wehbe and W.~Youssef.
\newblock Stabilization of the uniform {T}imoshenko beam by one locally
  distributed feedback.
\newblock {\em Applicable Analysis}, 88(7):1067--1078, 2009.

\bibitem{Wehbe01}
A.~Wehbe and W.~Youssef.
\newblock \href{http://dx.doi.org/10.1063/1.3486094}{Exponential and polynomial
  stability of an elastic {B}resse system with two locally distributed
  feedbacks}.
\newblock {\em J. Math. Phys.}, 51(10):103523, 17, 2010.

\bibitem{ZhangZuazua01}
X.~Zhang and E.~Zuazua.
\newblock \href{http://dx.doi.org/10.1016/j.jde.2004.02.004} {Polynomial decay
  and control of a {$1-d$} hyperbolic-parabolic coupled system}.
\newblock {\em J. Differential Equations}, 204(2):380--438, 2004.

\end{thebibliography}
\end{document}